\documentclass[a4paper,oneside,11pt]{article}
 
\usepackage[a4paper]{geometry}
\usepackage{aeguill}                
\usepackage{graphicx}
\usepackage{amsmath,amsfonts,amssymb,amsthm,amscd}
\usepackage{enumerate}
\usepackage{pstricks} 
\usepackage{epstopdf}
\usepackage{srcltx}
\usepackage{comment}
\usepackage{hyperref}
\usepackage{tikz-cd}
\usepackage[utf8]{inputenc} 
\usepackage[T1]{fontenc} 

\newtheorem{definition}{Definition} [section]
\newtheorem{de}[definition]{Definition}
\newtheorem{theorem}[definition]{Theorem} 
\newtheorem{theo}[definition]{Theorem}
\newtheorem{proposition}[definition]{Proposition}
\newtheorem{prop}[definition]{Proposition}
\newtheorem{lemma}[definition]{Lemma}

\newtheorem{remark}[definition]{Remark}

\newtheorem*{corintro}{Theorem~1}
\newtheorem*{corintro2}{Theorem~2}
\newtheorem*{propintro}{Proposition}

\def\aut{{\rm{Aut}}}
\def\out{{\rm{Out}}}
\def\G{\Gamma}

\def\stab{{\rm{Stab}}}

\newcommand{\suchthat}{\, | \,}
\newcommand{\Out}{{\rm{Out}}}
\newcommand{\Comm}{\rm{Comm}}
\newcommand{\Stab}{{\rm{Stab}}}
\newcommand{\calf}{\mathcal{F}}
\newcommand{\calI}{\mathcal{I}}
\newcommand{\calJ}{\mathcal{J}}

\newcommand{\caly}{\mathcal{Y}}

\newcommand{\Aut}{{\rm{Aut}}}
\newcommand{\ad}{{\rm{ad}}}

\newcommand{\Mod}{\rm{Mod}}
\newcommand{\zmax}{\mathcal{Z}_{\mathrm{RC}}}
\newcommand{\diam}{\mathrm{diam}}
\newcommand{\cali}{\calI}
\newcommand{\calj}{\calJ}
\newcommand{\calh}{\mathcal{H}}

\newcommand{\lk}{\rm{lk}}
\newcommand{\AT}{\mathcal{AT}}

\newcommand{\Inc}{\mathrm{Inc}}
\newcommand{\ia}{\mathrm{IA}_N(\mathbb{Z}/3\mathbb{Z})}

\newcommand{\ian}{{\mathrm{IA}_N}}
\newcommand{\Pcomp}{(P_{\mathrm{comp}})}
\newcommand{\Pstab}{(P_{\mathrm{stab}})}
\newcommand{\Ppcomp}{(P'_{\mathrm{comp}})}
\newcommand{\Pscomp}{(P''_{\mathrm{comp}})}

\newcommand{\ens}{\mathrm{FS}^{ens}}
\newcommand{\FF}{\mathrm{FF}}
\newcommand{\pr}{\mathrm{rk}_{\mathrm{prod}}}

\newcommand{\psl}{\mathrm{PSL}_n(\mathbb{Z})}
\newcommand{\NS}{\mathrm{FS}^{ns}}

\title{Commensurations of subgroups of $\Out(F_N)$}
\author{Camille Horbez~~and Richard D. Wade}
\date{}

\begin{document}

\maketitle

\begin{abstract}
A theorem of Farb and Handel \cite{FH} asserts that for $N\ge 4$, the natural inclusion from $\Out(F_N)$ into its abstract commensurator is an isomorphism. We give a new proof of their result, which enables us to generalize it to the case where $N=3$. More generally, we give sufficient conditions on a subgroup $\Gamma$ of $\Out(F_N)$ ensuring that its abstract commensurator $\Comm(\Gamma)$ is isomorphic to its relative commensurator in $\Out(F_N)$. In particular, we prove that the abstract commensurator of the Torelli subgroup $\ian$ for all $N\ge 3$, or more generally any term of the Andreadakis--Johnson filtration if $N\ge 4$, is equal to $\Out(F_N)$. Likewise, if $\Gamma$ the kernel of the natural map from $\Out(F_N)$ to the outer automorphism group of a free Burnside group of rank $N\geq 3$, then the natural map $\Out(F_N)\to\Comm(\Gamma)$ is an isomorphism.
\end{abstract}

\setcounter{tocdepth}{1}
\tableofcontents

\section*{Introduction}

Consider the following three classes of groups: the group $\psl$, with $n\ge 3$; the mapping class group $\Mod(\Sigma_g)$ of a closed, connected, oriented surface of genus $g\ge 2$; the group $\Out(F_N)$ of outer automorphisms of a finitely generated free group, with $N\ge 4$. All these groups are known to satisfy strong rigidity properties. For instance, if $\Gamma$ is a finite-index subgroup in one of them, then $\Out(\Gamma)$ is finite -- in other words $\Gamma$ has basically no more symmetries than the obvious ones given by conjugation within the ambient group. This is a consequence of the Mostow--Prasad--Margulis rigidity theorem \cite{Mos,Mos2,Pra,Mar} for $\psl$  (we are simplifying in this introduction by restricting our attention to $\psl$, but this discussion applies to many more lattices in semisimple Lie groups), was proved by Ivanov \cite{Iva} for mapping class groups, and by Farb--Handel \cite{FH} for $\Out(F_N)$, generalizing an earlier result of Khramtsov \cite{Khr} and Bridson--Vogtmann \cite{BV} stating that $\Out(\Out(F_N))$ is trivial for $N\ge 3$. 

A natural problem is to relax the symmetries one is allowed to look for, and study commensurations of the above groups instead of solely their automorphisms. Given a group $G$, the \emph{abstract commensurator} $\Comm(G)$ is the group whose elements are equivalence classes of isomorphisms between finite-index subgroups of $G$. The equivalence relation is given by saying that two such isomorphisms are equivalent if they agree on some common finite-index subgroup of their domains. Notice that every automorphism of $G$ determines an element of $\Comm(G)$; in particular, the action of $G$ on itself by conjugation gives a natural map $G\to\Comm(G)$. But in general, the abstract commensurator of a group $G$ is much larger than its automorphism group: for instance, the abstract commensurator of $\mathbb{Z}^n$ is isomorphic to $\mathrm{GL}(n,\mathbb{Q})$, and the abstract commensurator of a nonabelian free group is not finitely generated (see, for example \cite{MR2736164}).  Two groups $G$ and $H$ are \emph{abstractly commensurable} if they have isomorphic finite index subgroups. There is also a notion of \emph{relative commensurator}: given a group $G$ and a subgroup $H\subseteq G$, the \emph{relative commensurator} of $H$ in $G$, denoted as $\Comm_G(H)$, is the subgroup of $G$ made of all elements such that $H\cap gHg^{-1}$ has finite index in both $H$ and $gHg^{-1}$. There is always a natural  map $\Comm_G(H)\to\Comm(H)$. 

The Mostow--Prasad--Margulis rigidity theorem shows that the abstract commensurator of  $\psl$ is abstractly commensurable to its relative commensurator in $\mathrm{PGL}_n(\mathbb{R})$. Using work of Borel \cite{Bor}, this is known in turn to be isomorphic to $\mathrm{PGL}_n(\mathbb{Q})$, so the abstract commensurator is much larger than the automorphism group in this case. 

Mapping class groups and automorphism groups of free groups satisfy an even stronger form of rigidity. Ivanov proved in \cite{Iva} that for all $g\ge 3$, the natural map $\Mod^\pm(\Sigma_g)\to\Comm(\Mod(\Sigma_g))$ is an isomorphism. Farb and Handel proved in \cite{FH} that for every $N\ge 4$, the natural map $\Out(F_N)\to\Comm(\Out(F_N))$ is an isomorphism. In fact, every isomorphism between two finite-index subgroups of $\Out(F_N)$ extends to an inner automorphism of $\Out(F_N)$. Informally, these results imply that mapping class groups and $\Out(F_N)$ do not have natural enveloping `Lie groups'. These strong rigidity results have recently been extended to other groups, such as handlebody groups \cite{Hen} and big mapping class groups \cite{BDR}.

Margulis' normal subgroup theorem tells us that $\psl$ does not have a normal subgroup of infinite index. On the contrary, mapping class groups and $\Out(F_N)$ have many interesting normal subgroups. Ivanov's theorem has since been generalized to show that the abstract commensurator of various natural normal subgroups of $\Mod(\Sigma_g)$ is isomorphic to $\Mod^\pm(\Sigma_g)$. This includes the Torelli group \cite{FI} (with a recent extension to big mapping class groups in \cite{AGKMTW}), or more generally the further terms of the Johnson filtration \cite{BM2,Kid, BPS}. The latest development is a result by Brendle and Margalit \cite{BM}, asserting that if $\Gamma$ is a normal subgroup of $\Mod(\Sigma_g)$ that contains a `small' element (roughly, a homeomorphism supported on at most one third of the surface), then the natural map $\Mod^\pm(\Sigma_g)\to\Comm(\Gamma)$ induced by conjugation is an isomorphism. We warn the reader that the condition on `small' elements cannot be removed, as $\Mod(\Sigma_g)$ also contains normal purely pseudo-Anosov free subgroups \cite{DGO}, and as recalled earlier the abstract commensurator of a nonabelian free group is not finitely generated.  

Similarly to mapping class groups, $\out(F_N)$ also has many interesting normal subgroups, for instance $\ian$, which is the kernel of the action of $\out(F_N)$ on the abelianization of $F_N$. This is the first term in a family of normal subgroups called the Andreadakis--Johnson filtration, where the $k^{\text{th}}$ term is the kernel of the natural map from $\out(F_N)$ to the outer automorphism group of the free nilpotent group of rank $N$ of class $k$. The main result of the present paper is the following (we give a slightly weaker statement in rank $3$ just below).

\begin{corintro}\label{cor:intro}
Let $N\ge 4$, and let $\Gamma$ be either 
\begin{itemize}
\item a subgroup of $\Out(F_N)$ which contains a term of the Andreadakis--Johnson filtration of $\Out(F_N)$, or 
\item a subgroup of $\Out(F_N)$ that contains a power of every Dehn twist. 
\end{itemize}
Then the natural map $\Comm_{\Out(F_N)}(\Gamma)\to\Comm(\Gamma)$ is an isomorphism. In fact, every isomorphism between two finite-index subgroups of $\Gamma$ is equal to the restriction of the conjugation by some element in $\Comm_{\Out(F_N)}(\Gamma)$.
\end{corintro} 

In rank three, we prove the following.

\begin{corintro2}
Let $\Gamma$ be either $\mathrm{IA}_3$ or a subgroup of $\Out(F_3)$ that contains a power of every Dehn twist. 
\\ Then the natural map $\Comm_{\Out(F_3)}(\Gamma)\to\Comm(\Gamma)$ is an isomorphism. In fact, every isomorphism between two finite-index subgroups of $\Gamma$ is equal to the restriction of the conjugation by some element in $\Comm_{\Out(F_3)}(\Gamma)$.
\end{corintro2}

\noindent\emph{Example.} Let $N\ge 3$, let $p\in\mathbb{N}$, and let $\Gamma$ be the kernel of the natural map from $\out(F_N)$ to the outer automorphism group of the free Burnside group $B(N,p)$. Then $\Gamma$ contains the $p^{\text{th}}$ power of every Dehn twist, and hence is covered by the theorem. As $\Gamma$ is normal in $\out(F_N)$, we deduce that the natural map $\Out(F_N)\to\Comm(\Gamma)$ is an isomorphism.
\\
\\
\indent Let us make a few more comments about our main theorem. First, we recover Farb and Handel's theorem that $\Comm(\Out(F_N))\simeq\Out(F_N)$ -- with a new proof -- and extend it to the case where $N=3$. Second, in the case where $\Gamma$ is normal, the conclusion is that the natural map $\Out(F_N)\to\Comm(\Gamma)$ is an isomorphism, so our theorem computes the abstract commensurator of $IA_N$ and of all terms in the Andreadakis--Johnson filtration if $N\ge 4$. Third, the requirement that $N \geq 3$ in the above theorem is strict as the group $\Out(F_2)$ is virtually free and therefore has a more complicated abstract commensurator. Finally, we would like to mention that when $N\ge 4$, all examples in the statement are recast in the more general framework of \emph{twist-rich} subgroups of $\Out(F_N)$ (see Section~\ref{sec:hypotheses} for the precise definition of twist-rich and Section~\ref{sec:conclusion} for the most general statement of Theorem~1\ref{cor:intro}).

While Farb and Handel's proof in \cite{FH} was more algebraic (and relied on previous work of Feighn and Handel \cite{FeH} classifying abelian subgroups of $\Out(F_N)$), the broad strategy of our proof is closer in spirit to Ivanov's, which relied on the computation of the symmetries of the curve complex. Namely, we use the fact that the simplicial automorphisms of a certain $\Out(F_N)$-complex all come from the $\Out(F_N)$-action. Before giving a simplified sketch of the proof, we feel that is worth highlighting three places where, as far as we are aware, our techniques differ from the current literature: 
\begin{itemize}
\item We provide a general framework in the language of relative commensurators, which allows us to understand $\Comm(\G)$ for subgroups $\G$ of $\Out(F_N)$ that are not necessarily normal. 

\item As we shall see below the algebraic structure of $\Out(F_N)$ is quite different from a mapping class group, and this is used in the proof in an essential way. In particular, we will crucially take advantage of twist subgroups associated to one-edge free splittings, which do not have a natural analogue in the surface setting.
 
\item Actions of subgroups on Gromov hyperbolic spaces and their boundaries are a fundamental part of the proof. \end{itemize}

\paragraph*{Strategy of proof.} The rest of the introduction is devoted to sketching our proof that the natural map $\Out(F_N)\to\Comm(\Out(F_N))$ is an isomorphism for all $N\ge 3$ -- a few more technicalities arise for general twist-rich subgroups, but we will ignore them for now. As we are working up to commensuration, it is actually enough to compute the abstract commensurator of the torsion-free finite-index subgroup $\ia$ made of automorphisms acting trivially on homology mod $3$ (this is useful in order to avoid some finite-order phenomena). 

Various natural $\Out(F_N)$-complexes are known to be rigid in the sense that all their simplicial automorphisms come from the $\Out(F_N)$-action. These include the spine of reduced Outer space \cite{BV2}, the free splitting complex \cite{AS}, the complex of nonseparating free splittings \cite{Pan}, the cyclic splitting complex \cite{HW} or the free factor complex \cite{BB}. We work with the \emph{edgewise nonseparating free splitting graph} $\ens$, defined as follows: vertices are free splittings of $F_N$ as an HNN extension $F_N=A\ast$, and two splittings are joined by an edge if they are \emph{rose-compatible}, i.e.\ they have a common refinement which is a two-petalled rose (if they are compatible and their refinement is a two-edge loop, we do not add an edge). We prove that for $N\ge 3$, this graph is rigid in the above sense.

We then show that every commensuration $f$ of $\Out(F_N)$ induces a simplicial automorphism $f_*$ of $\ens$ -- once this is done, a general argument presented in Section~\ref{sec:blueprint} allows us to deduce that $f$ is induced by conjugation and therefore $\Comm(\Out(F_N))\simeq\Out(F_N)$. This comes in two parts: we need to define $f_*$ on the vertex set of $\ens$ and then we need to show that $f_*$ respects edges in $\ens$.

Firstly, we look at the vertex set of $\ens$. Each vertex is given by a nonseparating free splitting $S$ and its stabilizer has a finite index subgroup $H_S$ contained in the domain of $f$. We  give a purely algebraic characterization of $\Out(F_N)$-stabilizers of nonseparating free splittings. This will imply that there is a unique splitting $S'$ whose stabilizer in $\Out(F_N)$ contains $f(H_S)$, allowing us to define $f_*(S)=S'$. A short argument using the fact that $f$ is invertible implies that $f_*$ is a bijection on the vertex set and that $f(H_S)$ is finite index in the stabilizer of $S'$. The idea for the characterization is the following: the group of twists associated to the splitting $S$ is by \cite{Lev} a direct product of two nonabelian free groups isomorphic to $F_{N-1}$. This gives a direct product of free groups $K_1\times K_2$ which is normal in $H_S$. In addition, the centralizer of $K_i$ (or more generally, a normal subgroup of $K_i$) in $\Out(F_N)$ is a free group (the centralizer of $K_1$ is $K_2$ and vice versa). These features are enough for the characterization: we prove the following.

\begin{propintro}
Let $H$ be a subgroup of $\ia$ which contains a normal subgroup that splits as a direct product $K_1\times K_2$ of two nonabelian subgroups, such that for every $i\in\{1,2\}$, and every subgroup $P_i$ which is normal in a finite-index subgroup of $K_i$, the centralizer of $P_i$ in $\ia$ is equal to $K_{i+1}$ (where indices are taken mod $2$).
\\ Then $H$ fixes a free splitting of $F_N$.
\end{propintro}

An examination of maximal free abelian subgroups (or of maximal direct products of free groups inside $H$), then enables us to distinguish separating and nonseparating free splittings. The idea behind our proof of the above proposition is that containing a normal direct product restricts the possible actions of $H$ on a hyperbolic graph. We let $\calf$ be a maximal $H$-invariant free factor system of $F_N$, and apply this idea to the relative free factor graph $\FF:=\FF(F_N,\calf)$, which is known to be hyperbolic \cite{BF,HM3}. If this free factor graph has bounded diameter, then the free factor system is called \emph{sporadic}, and in this case the group $H$ fixes a free splitting. Otherwise $\FF$ has infinite diameter. Furthermore, a theorem of Guirardel and the first author \cite{GH} states that if $\calf$ is maximal, then $H$ acts on $\FF$ with unbounded orbits. It remains to show that if $H$ contains a normal $K_1 \times K_2$ as above then this cannot happen. Using Gromov's classification of group actions on hyperbolic spaces, we show that if $H$ acts on $\FF$ (or indeed, any hyperbolic graph) with unbounded orbits, then one of the subgroups $K_i$ has a finite orbit in the Gromov boundary $\partial_\infty \FF$. The boundary $\partial_\infty\FF$ has been identified \cite{BR,Ham,GH} with a space of \emph{arational} $(F_N,\calf)$-trees. The most technical work in the paper is an analysis of isometric stabilizers of arational trees, which relies on arguments of Guirardel and Levitt \cite{GL}: in particular, we show that they have a $\mathbb{Z}^2$ in their centralizer. This implies that an isometric stabilizer of an arational  tree cannot contain a normal subgroup of a $K_i$ as the centralizer of such a group is free. This finishes the proposition and allows for the definition of $f_*$ on the vertices.

To show this map $f_*$ extends to the edge set of $\ens$, we need to give an algebraic characterization of when two free splittings are compatible -- distinguishing between rose compatibility and circle compatibility can then be done algebraically by considering maximal abelian subgroups in the common stabilizer. The key idea is to observe that two one-edge free splittings $S$ and $S'$ are noncompatible if and only if their common stabilizer also fixes a third one-edge free splitting. Indeed, thinking of free splittings as spheres in a doubled handlebody, the stabilizer of two spheres that intersect also, up to finite index, fixes any sphere obtained by surgery between them. Conversely, if two nonseparating free splittings have a common refinement (or equivalently determine disjoint spheres), then their common stabilizer does not fix any other free splittings. This characterization shows that $f_*$ extends to the edge set of $\ens$, and concludes our proof.

\paragraph*{Organization of the paper.} The paper is organized as follows. In Sections~\ref{sec:blueprint} to~\ref{sec:product-f2}, we collect several tools that will be crucial in the proof of our main theorem: these include (in addition to general background on $\Out(F_N)$ given in Section~\ref{sec:background})
\begin{itemize}
\item a general framework to deduce commensurator rigidity from the rigidity of a graph (Section~\ref{sec:blueprint}),
\item a proof that the edgewise nonseparating free splitting graph is rigid (Section~\ref{sec:ens}),
\item an analysis of actions of direct products on hyperbolic spaces (Section~\ref{sec:direct-product-vs-hyp}),
\item an analysis of stabilizers of relatively arational trees (Section~\ref{sec:arat}),
\item an analysis of maximal direct products of free groups in $\Out(F_N)$ (Section~\ref{sec:product-f2}).
\end{itemize}
The next sections are devoted to the proof of rigidity. In Section~\ref{sec:hypotheses}, we define twist-rich subgroups of $\Out(F_N)$. In Section~\ref{sec:vertices}, we prove that the commensurability classes of vertex stabilizers of $\ens$ are $\Comm(\Gamma)$-invariant, and in Section~\ref{sec:edges} we prove the same thing for stabilizers of edges. This is enough to conclude the proof in Section~\ref{sec:conclusion}.

\paragraph*{Acknowledgements.} We would like to thank Martin Bridson and Vincent Guirardel for enlightening discussions about this project. In particular, Martin Bridson showed us that direct products of free groups and the special structure of stabilizers of nonseparating free splittings could be utilised in these rigidity problems. We would like to thank Vincent Guirardel for many discussions regarding the structure of stabilizers of arational trees.  We are grateful to Vincent Guirardel and Gilbert Levitt for sharing with us some arguments from their ongoing paper on stabilizers of $\mathbb{R}$-trees \cite{GL} and allowing us to use some of these arguments in Section~\ref{sec:arat} of the present paper. We would also like to thank the referee for their careful reading of the paper and helpful comments.

The first author acknowledges support from the Agence Nationale de la Recherche under Grant ANR-16-CE40-0006. He also thanks the Fields Institute, where this project was completed during the \emph{Thematic Program on Teichmüller Theory and its Connections to Geometry, Topology and Dynamics} in Fall 2018, for its hospitality.

\section{Commensurations and complexes}\label{sec:blueprint}

\emph{In this section, we setup a general framework  to use the rigidity of a graph equipped with an action of a group $G$ in order to compute the abstract commensurator of $G$ and some of its subgroups.}
\\
\\
\indent Let $G$ be a group. We recall from the introduction that the \emph{abstract commensurator} $\Comm(G)$ is the group whose elements are the equivalence classes of isomorphisms $f:H_1 \to H_2$ between finite index subgroups of $G$. The equivalence relation is given by saying that $f$ is equivalent to $f':H_1' \to H_2'$ if $f$ and $f'$ agree on some common finite index subgroup $H$ of their domains. We will denote by $[f]$ the equivalence class of $f$. The identity element of $\Comm(G)$ is the equivalence class of the identity map on $G$, and composition $[f]\cdot[f']$ is obtained by restriction to a finite index subgroup so that $f \circ f'$ is well-defined. Notice that if $H$ is finite index in $G$, then the natural map $\Comm(G)\to\Comm(H)$ (obtained by restriction to a further finite-index subgroup) is an isomorphism.

Two subgroups $P_1$ and $P_2$ of $G$ are \emph{commensurable in $G$} if their intersection $P_1\cap P_2$ has finite index in both $P_1$ and $P_2$. We will denote by $[P]$ the commensurability class of a subgroup $P$ of $G$. The group $\Comm(G)$ acts on the set of all commensurability classes of subgroups of $G$ by letting $[f]\cdot[P]=[f(P)]$, where $P$ is any representative of its commensurability class that is contained in the domain of $f$. 

We now let $\Gamma\subseteq G$ be a subgroup of $G$. We recall that the \emph{relative commensurator} of $\Gamma$ in $G$, denoted by $\Comm_G(\Gamma)$, is the subgroup of $G$ made of all elements $g$ such that $\Gamma$ and $g\Gamma g^{-1}$ are commensurable in $G$. In this case, if $\ad_g$ is the inner automorphism of $g$ sending $h \mapsto ghg^{-1}$ and $g \in \Comm_G(\Gamma)$, then $\ad_g$ restricts to an isomorphism between the finite index subgroups $g^{-1}\Gamma g \cap \Gamma$ and $\Gamma \cap g \Gamma g^{-1}$ of $\Gamma$. In this way, the action of $\Comm_G(\Gamma)$ by conjugation induces a map $\ad \colon \Comm_G(\Gamma) \to \Comm(\Gamma)$. 

In the following statement, given a graph $X$, we let $V(X)$ be the vertex set of $X$ and $E(X)$ be the edge set of $X$. We use $\Aut(X)$ to denote the group of graph automorphisms of $X$.  A graph $X$ is \emph{simple} if it contains no edge-loops and there are no multiple edges between pairs of vertices. This is equivalent to the condition that every automorphism of $X$ is determined by its induced map on the vertices.

\begin{prop}\label{prop:blueprint}
Let $G$ be a group, let $\Gamma\subseteq G$ be a subgroup. Let $X$ be a simple graph equipped with a $G$-action by graph automorphisms. Assume that 
\begin{enumerate}
\item the natural map $G\to\Aut(X)$ is an isomorphism,
\item given two distinct vertices $v$ and $w$ in $X$, the groups $\Stab_\Gamma(v)$ and $\Stab_\Gamma(w)$ are not commensurable in $\Gamma$,
\item the sets $$\cali:=\{[\Stab_\Gamma(v)]\suchthat v\in V(X)\}$$ and $$\calj:=\{([\Stab_\Gamma(v)],[\Stab_\Gamma(w)])\suchthat vw\in E(X)\}$$ are $\Comm(\Gamma)$-invariant (in the latter case with respect to the diagonal action).
\end{enumerate}
\noindent Then any isomorphism $f\colon H_1 \to H_2$ between two finite index subgroups of $\G$ is given by conjugation by an element of $\Comm_G(\G)$ and $\ad\colon\Comm_G(\Gamma)\to\Comm(\Gamma)$ is an isomorphism. 
\end{prop}

\begin{proof}
We first define a map $\Phi:\Comm(\Gamma)\to\Aut(X)$ in the following way. As $\cali$ is $\Comm(\Gamma)$-invariant, given  $f\in\Comm(\Gamma)$ and a vertex $v\in X$, there exists a vertex $w\in X$ such that $f([\Stab_\Gamma(v)])=[\Stab_{\Gamma}(w)]$; in addition, our second hypothesis ensures that this vertex $w$ is unique. We thus get a map $V(X)\to V(X)$, sending $v$ to $w$, and this map is bijective because $f$ is invertible. As two vertices of $X$ are adjacent if and only if $([\Stab_\Gamma(v)],[\Stab_{\Gamma}(w)])\in\calj$, and $\calj$ is $\Comm(\Gamma)$-invariant, the above map extends to a graph automorphism of $X$. Hence $\Phi$ is well-defined, and it is easy to check that $\Phi$ is a homomorphism. From now on, given $f\in\Comm(\Gamma)$, we will let $f_X:=\Phi(f)$ denote the induced action on $X$.

Let $\Psi:G\to\Aut(X)$ be the natural map. We next claim that the following diagram commutes:
\[
  \begin{tikzcd}
 G\arrow[bend left=10]{rrd}{\Psi} & & \\ 
    \Comm_G(\Gamma) \arrow{r}{\ad}\arrow[u, hook] & \Comm(\Gamma) \arrow{r}{\Phi} & \aut(X).
      \end{tikzcd}
\]    
Equivalently, we need to check that if $g\in\Comm_G(\Gamma)$ and $v\in X$, then $(\ad_g)_X(v)=gv$. This holds because: 
\begin{displaymath}
\begin{array}{rlc}
\ad_g([\Stab_\Gamma(v)])&=\ad_g([\Stab_G(v)\cap\Gamma]) &\\
&=[\Stab_G(gv)\cap\Gamma] &\text{~as $g\in\Comm_G(\Gamma)$} \\
&=[\Stab_\Gamma(gv)]. &
\end{array}     
\end{displaymath}
Now let $f \colon H_1\to H_2$ be an isomorphism between two finite-index subgroups of $\Gamma$. Then $[f]_X=\Psi(g)$ for some $g \in G$ as $\Psi$ is surjective. We aim to prove that $f$ is equal to the restriction to $H_1$ of the conjugation by $g$ in $G$: this will imply in particular that $g \in \Comm_G(\G)$, and that $f=\ad_g$ in $\Comm(\Gamma)$. Let $h \in H_1$. Then \[ [\ad_{f(h)}] = [f\circ\ad_h\circ f^{-1}] \] in $\Comm(\Gamma)$, therefore  \[ [\ad_{f(h)}]_{X}=\Psi(g) \circ [\ad_{h}]_{X} \circ \Psi(g^{-1})\] as $[f]_X=\Psi(g)$. By commutativity of the diagram we have $[\ad_{f(h)}]_{X}=\Psi(f(h))$ and $[\ad_{h}]_{X}=\Psi(h)$, so that $\Psi(f(h))=\Psi(ghg^{-1})$. As $\Psi$ is injective, this implies that $f(h)=ghg^{-1}$, as desired. This shows that the map $\ad \colon \Comm_G(\G) \to \Comm(\G)$ is surjective. It is also injective as the diagram commutes and the top two arrows are injective.
\end{proof}

\section{Background on $\Out(F_N)$}\label{sec:background}

\emph{In this section, we review some general background on $\Out(F_N)$. In particular we look at the geometry of relative free factor complexes, and establish a few basic facts about Dehn twist automorphisms.}

\subsection{Splittings and free factor systems}

A \emph{splitting} of $F_N$ is a minimal, simplicial $F_N$-action on a simplicial tree $S$ (we recall that the action is said to be \emph{minimal} if $S$ does not contain any proper $F_N$-invariant subtree). Splittings of $F_N$ are always considered up to $F_N$-equivariant homeomorphism. A \emph{free splitting} of $F_N$ is a splitting of $F_N$ in which all edge stabilizers are trivial. A $\mathcal{Z}_{max}$ splitting of $F_N$ is a splitting of $F_N$ in which all edge stabilizers are isomorphic to $\mathbb{Z}$ and root-closed. A \emph{$\zmax$ splitting} of $F_N$ is a splitting of $F_N$ in which all edge stabilizers are either trivial or isomorphic to $\mathbb{Z}$ and root-closed. The class of $\zmax$ splittings contains all free splittings and all $\mathcal{Z}_{max}$ splittings. We say that a splitting is a \emph{one-edge} splitting if the quotient graph $S/F_N$ consists of a single edge, and a \emph{loop-edge} splitting if $S/F_N$ is a single loop. We say that a splitting $S'$ is a \emph{blowup} or, equivalently, a \emph{refinement} of $S$ if $S$ is obtained from $S'$ by collapsing some edge orbits in $S'$. The splitting $S'$ is a \emph{blowup of $S$ at a vertex $v\in S$} if every collapsed edge from $S'$ has its image in the $F_N$-orbit of $v$ under the quotient map $S'\to S$.  Two splittings are \emph{compatible} if they admit a common refinement.

A \emph{free factor system} of $F_N$ is a collection $\calf$ of conjugacy classes of subgroups of $F_N$ which arise as the collection of all nontrivial point stabilizers in some nontrivial free splitting of $F_N$. Equivalently, this is a collection of conjugacy classes of subgroups $A_i$ such that $F_N$ splits as $F_N=A_1\ast\dots\ast A_k\ast F_r$. We sometimes blur the distinction between the finite set of all conjugacy classes in $\calf$ and the infinite set of all free factors whose conjugacy classes belong to $\calf$. The free factor system is \emph{sporadic} if $(k+r,r)\le (2,1)$ (for the lexicographic order), and \emph{nonsporadic} otherwise. Concretely, the sporadic free factor systems are those of the form $\{[C]\}$ where $C$ is rank $N-1$ so that $F_N=C\ast$, and those of the form $\{[A],[B]\}$ where $F_N=A\ast B$. The collection of all free factor systems of $F_N$ has a natural partial order, where $\calf\leq\calf'$ if every factor in $\calf$ is conjugate into one of the factors in $\calf'$.

More generally, if $\calh$ is a collection of conjugacy classes of subgroups of $F_N$, there exists a unique smallest free factor system $\calf$ of $F_N$ such that every subgroup in $\calh$ is conjugate into a subgroup of $\calf$. We say that the pair $(F_N,\calh)$ is \emph{sporadic} if $\calf$ is sporadic.  

Given a free factor system $\calf$ of $F_N$, a \emph{free splitting of $F_N$ relative to $\calf$} is a free splitting of $F_N$ in which every factor in $\calf$ fixes a point. A \emph{free factor} of $(F_N,\calf)$ is a subgroup of $F_N$ which arises as a point stabilizer in some free splitting of $F_N$ relative to $\calf$. A free factor is \emph{proper} if it is nontrivial, not conjugate to an element of $\calf$ and not equal to $F_N$.  An element $g\in F_N$ is \emph{peripheral} (with respect to $\calf$) if it is conjugate into one of the subgroups in $\calf$, and \emph{nonperipheral} otherwise. 

Given a free factor system $\calf$, we denote by $\Out(F_N,\calf)$ the subgroup of $\Out(F_N)$ made of all automorphisms that preserve the conjugacy classes of free factors in $\calf$. Given a subgroup $H\subseteq\Out(F_N)$, we say that $\calf$ is \emph{$H$-periodic} if the $H$-orbit of $\calf$ is finite, equivalently if $H$ has a finite-index subgroup contained in $\Out(F_N,\calf)$ (the notions of $H$-periodic free factors and free splittings are defined in the same way).

\subsection{Relative free factor graphs}

\paragraph*{The definition of $\FF(F_N, \calf)$ and hyperbolicity.} Given a free factor system $\calf$ of $F_N$, the \emph{free factor graph} $\FF(F_N,\calf)$ is the graph whose vertices are the nontrivial free splittings of $F_N$ relative to $\calf$, where two free splittings are joined by an edge if they are either compatible or share a nonperipheral elliptic element. In this way $\FF(F_N,\calf)$ is defined as an electrification of another natural $\Out(F_N,\calf)$-graph, the so-called \emph{free splitting graph}. This definition of the free factor graph, which is the one adopted in \cite{GH}, has the advantage of being adapted to all nonsporadic free factor systems $\calf$. Except in some low-complexity cases, it is quasi-isometric to all other models of the free factor graph available in the literature (e.g. where vertices are given by proper free factors of $(F_N,\calf)$), as discussed in \cite[Section~2.2]{GH}. The free factor graph $\FF(F_N,\calf)$ is always hyperbolic: this was first proved by Bestvina and Feighn \cite{BF} in the crucial absolute case where $\calf=\emptyset$, and then extended by Handel and Mosher \cite{HM3} to the general case (with the exception of one low-complexity case which is handled in \cite[Proposition~2.11]{GH}). 

\paragraph*{The boundary of $\FF(F_N,\calf)$.} We will now recall the description of the Gromov boundary of $\FF(F_N,\calf)$ in terms of certain $F_N$-actions on $\mathbb{R}$-trees \cite{BR,Ham,GH}.

An \emph{$(F_N,\calf)$-tree} is an $\mathbb{R}$-tree $T$ equipped with a minimal isometric $F_N$-action in which every subgroup in $\calf$ fixes a point. It is a \emph{Grushko $(F_N,\calf)$-tree} if $T$ is a simplicial metric tree, and every nontrivial point stabilizer in $T$ is conjugate to an element of $\calf$. When $\calf=\emptyset$, the space of all Grushko $F_N$-trees is nothing but Culler and Vogtmann's Outer space $CV_N$ from \cite{CV}. 

Given an $F_N$-action on an $\mathbb{R}$-tree $T$ and a subgroup $A\subseteq F_N$ which does not fix a point in $T$, there exists a unique minimal $A$-invariant subtree of $T$ (which is equal to the union of all axes of elements of $A$ acting hyperbolically on $T$). This is normally denoted by $T_A$.

If $A$ is a proper free factor of $(F_N,\calf)$ then there is an associated free factor system $\calf_{|A}$  of $A$ given by the vertex stabilizers appearing in the action of $A$ on a Grushko $(F_N,\calf)$-tree. An $(F_N,\calf)$-tree $T$ is \emph{arational} if $T$ is not a Grushko $(F_N,\calf)$-tree, no proper $(F_N,\calf)$-free factor fixes a point in $T$, and for every proper $(F_N,\calf)$-free factor $A$, the $A$-action on its minimal invariant subtree $T_A\subseteq T$ is a Grushko $(A,\calf_{|A})$-tree. We denote by $\AT(F_N,\calf)$ the space of all arational $(F_N,\calf)$-trees, equipped with the equivariant Gromov--Hausdorff topology introduced in \cite{Pau}. Two arational trees $T$ and $T'$ are \emph{equivalent} (denoted as $T\sim T'$) if they admit $F_N$-equivariant alignment-preserving bijections to one another. Arational trees are used to describe the boundary of the free factor graph: the following theorem was established by Bestvina and Reynolds \cite{BR} and independently Hamenstädt \cite{Ham} in the case where $\calf=\emptyset$, and extended by Guirardel and the first author in \cite{GH} to the general case.

\begin{theo}
Let $\calf$ be a nonsporadic free factor system of $F_N$. Then there exists an $\Out(F_N,\calf)$-equivariant homeomorphism $\AT(F_N,\calf)/{\sim}\to\partial_\infty\FF(F_N,\calf)$.
\end{theo}

We also mention that the space of all projective classes of arational trees in a given $\sim$-class is a finite-dimensional simplex, see e.g.\ \cite[Proposition~13.5]{GH1}. In particular, we record the following fact.

\begin{prop}\label{prop:fix-boundary}
Let $\calf$ be a nonsporadic free factor system of $F_N$, and let $H\subseteq\Out(F_N,\calf)$ be a subgroup which has a finite orbit in $\partial_\infty\FF(F_N,\calf)$.
\\ Then $H$ has a finite-index subgroup that fixes the homothety class of an arational $(F_N,\calf)$-tree.
\end{prop}

In the rest of the paper, a \emph{relatively arational tree} will be an $F_N$-tree which is arational relative to some (nonsporadic) free factor system of $F_N$.

\paragraph*{Dynamics of subgroups of $\Out(F_N,\calf)$ acting on $\FF(F_N,\calf)$.} It will be important to determine whether certain subgroups of $\Out(F_N,\calf)$ have bounded or unbounded orbits in the relative free factor graph. To this end, we will use the following theorem established by Guirardel and the first author in \cite[Proposition~5.1]{GH}. 

\begin{theorem}
Let $\calf$ be a nonsporadic free factor system, and let $H\subseteq\Out(F_N)$ be a subgroup which acts on $\FF(F_N,\calf)$ with bounded orbits. Then there exists a $H$-periodic free factor system $\calf'$ such that $\calf\le\calf'$ and $\calf' \neq \calf$.
\end{theorem}

While working with subgroups of $\out(F_N)$, it is convenient to have factor systems that are genuinely fixed rather than just being periodic. For this reason, it is good to work in the group $\ia$, which is the finite-index subgroup of $\Out(F_N)$ defined as the kernel of the natural map $\Out(F_N)\to\mathrm{GL}_N(\mathbb{Z}/3\mathbb{Z})$ given by the action on $H_1(F_N;\mathbb{Z}/3\mathbb{Z})$. It satisfies a certain number of useful properties, of particular importance being:

\begin{theo}[{Handel--Mosher \cite[Theorem~3.1]{HM5}}]
Let $H\subseteq\ia$ be a subgroup, and let $A\subseteq F_N$ be a free factor whose conjugacy class is $H$-periodic.
\\ Then the conjugacy class of $A$ is $H$-invariant.
\end{theo}

As noted in the previous section, passing to a finite index subgroup does not change the abstract commensurator of a group, and for this reason we work in $\ia$ for much of the paper. Handel and Mosher's theorem implies that if $H$ is contained in $\ia$ then a $H$-periodic free factor system $\calf$ is $H$-invariant. Combining both of the above results gives:

\begin{proposition}\label{prop:maximal-unbounded}
Suppose that $H$ is a subgroup of $\ia$ and $\calf$ is a maximal, $H$-invariant free factor system. If $\calf$ is not sporadic, then $H$ acts on $\FF(F_N,\calf)$ with unbounded orbits.
\end{proposition}

For future use, we also mention another fact about $\ia$ that we will use several times in the paper.

\begin{lemma}\label{lemma:stab-splitting-ia}
Let $S$ be a free splitting of $F_N$. Let $H\subseteq\ia$ and suppose that $S$ is $H$-periodic.
\\ Then $H \subseteq \Stab(S)$ and $H$ acts trivially on the quotient graph $S/F_N$.
\\ In particular, if $\hat{S}$ is a refinement of $S$, then $\Stab_{\ia}(\hat{S})\subseteq\Stab_{\ia}(S)$. 
\end{lemma}

\begin{proof}
 The second conclusion of the lemma is a consequence of the first, so we focus on the first. Each one-edge free splitting of $F_N$ is determined by a sporadic free factor system, so by the theorem of Handel and Mosher, any one-edge splitting that is periodic under $H$ is in fact invariant. In general, an arbitrary splitting $S$ is determined by its one-edge collapses, so if $S$ is $H$-periodic and $H\subseteq\ia$ then $S$ is $H$-invariant. This argument also shows that $H$ preserves the edges in $S/F_N$, and will act trivially on $S/F_N$ if no edges are flipped. Such a flip is visible in $H_1(F_N;\mathbb{Z}/3\mathbb{Z})$ as either it induces a nontrivial action on $H_1(S/F_N;\mathbb{Z}/3\mathbb{Z})$ (if the splitting is nonseparating), or it permutes distinct free factors (if the splitting is separating).
\end{proof}

\subsection{Groups of twists}

Let $S$ be a splitting of $F_N$, let $v\in S$ be a vertex, let $e$ be a half-edge of $S$ incident on $v$, and let $z$ be an element in $C_{G_v}(G_e)$ (the centralizer of the stabilizer of $e$ inside the stabilizer of $v$; notice in particular that the existence of such a $z$ implies that $G_e$ is either trivial or cyclic). Following \cite{Lev}, we define the \emph{twist by $z$ around $e$} to be the automorphism $D_{e,z}$ of $F_N$ (preserving $S$) defined in the following way. Let $\overline{S}$ be the splitting obtained from $S$ by collapsing all half-edges outside of the orbit of $e$; we denote by $\overline{e}$ (resp.\ $\overline{v}$) the image of $e$ (resp.\ $v$) in $\overline{S}$, and by $\overline{w}$ the other extremity of $\overline{e}$. If the extremities of $\overline{e}$ are in distinct $F_N$-orbits, then we have an amalgam, and $D_{e,z}$ is defined to be the unique automorphism that acts as the identity on $G_{\overline{v}}$, and as conjugation by $z$ on $G_{\overline{w}}$. 
If the extremities $\overline{v},\overline{w}$ of $\overline{e}$ are in the same $F_N$-orbit, then we let $t\in F_N$ be such that $\overline{w}=t\overline{v}$, and $D_{e,z}$ is defined as the identity on $G_{\overline{v}}$, with $D_{e,z}(t)=zt$. In this case, $D_{e,z}$ is a Nielsen automorphism. 
The element $z$ is called the \emph{twistor} of $D_{e,z}$. The \emph{group of twists} of the splitting $S$ is the subgroup of $\Out(F_N)$ generated by all twists around half-edges of $S$. 

\paragraph*{Twists about cyclic splittings.} 
Let $S$ be a splitting of $F_N$ with exactly one orbit of edges, whose stabilizer is root-closed and isomorphic to $\mathbb{Z}$. Then the group of twists of the splitting $S$ is isomorphic to $\mathbb{Z}$ (see \cite[Proposition~3.1]{Lev}).

\begin{lemma}\label{twist-compatible}
Let $S$ be a splitting of $F_N$ with exactly one orbit of edges whose stabilizer is root-closed and isomorphic to $\mathbb{Z}$, and let $D$ be a nontrivial twist about $S$. Let $R$ be a free splitting of $F_N$, such that $D(R)=R$.
\\ Then $S$ and $R$ are compatible.
\end{lemma}

\begin{proof}
The key tool in the proof is a theorem of Cohen and Lustig \cite{CL}, which shows that every free action of $F_N$ on an $\mathbb{R}$-tree has a parabolic orbit in Outer space which converges to a defining tree for the twist.

Let $\hat{R}$ be a simplicial metric $F_N$-tree in unprojectivized Outer space $cv_N$ that collapses onto $R$. By \cite{CL}, there exists a sequence $(\lambda_n)_{n\in\mathbb{N}}\in(\mathbb{R}_+^\ast)^\mathbb{N}$ such that $\lambda_nD^{n}(\hat{R})$ converges to $S$ (in the Gromov--Hausdorff equivariant topology). Since for every $n\in\mathbb{N}$, the splittings $\lambda_n D^{n}(\hat{R})$ and $R=D^{n}(R)$ are compatible, it follows from \cite[Corollary~A.12]{GL-jsj} that in the limit, the splittings $S$ and $R$ are compatible.
\end{proof}

Given a subgroup $K$ of $\Out(F_N)$, we denote by $C_{\out(F_N)}(K)$ the centralizer of $K$ in $\Out(F_N)$. More generally, if $H$ is a subgroup of $\Out(F_N)$ then we use $C_H(K)$ to denote the intersection of the centralizer with $H$. Twists determined by cyclic edges are central in a finite index subgroup of the stabilizer of the tree:

\begin{lemma}[{Cohen--Lustig \cite[Lemma~5.3]{CL2}}]\label{twist-central}
Let $S$ be a splitting of $F_N$ with exactly one orbit of edges whose stabilizer is isomorphic to $\mathbb{Z}$, and let $D$ be a nontrivial twist about $S$. Let $H_S$ be the subgroup of $\Out(F_N)$ that stabilizes $S$, acts as the identity on the quotient graph $S/F_N$, and induces the identity on each of the edge groups of $S$. Then $D$ is central in $H_S$.
\end{lemma}

We establish one more fact about twists about cyclic splittings. 

\begin{lemma}\label{lemma:twistor}
Let $S$ be a splitting of $F_N$ with exactly one orbit of edges whose stabilizer is isomorphic to $\mathbb{Z}$, and let $w$ be a generator of the edge group of $S$. Let $\Phi\in\Out(F_N)$ be an automorphism which commutes with the Dehn twist about $S$.
\\ Then $\Phi$ preserves the conjugacy class of $\langle w\rangle$.
\end{lemma}

\begin{proof}
Let $\sqrt{w}$ be the unique smallest root of $w$ (in particular, $w=\sqrt{w}^k$ for some $k \geq 1$). We can replace $S$ with a splitting $S'$ with edge stabilizer $\sqrt{w}$ by equivariantly folding an edge $e$ with $\sqrt{w}\cdot e$ (Cohen and Lustig describe this process as \emph{getting rid of proper powers} \cite{CL2}). The tree $S'$ has an edge group generated by $\sqrt{w}$. Any Dehn twist on $S$ is also a Dehn twist on $S'$, and furthermore Cohen and Lustig's parabolic orbit theorem implies that the centralizer of a such a Dehn twist fixes the splitting $S'$ (see, for example, \cite[Corollary~6.8]{CL2}). As $\Phi \cdot S'=S'$ and there is only one orbit of edges in $S'$, the conjugacy class of $\langle \sqrt w \rangle$ (and therefore the conjugacy class of $\langle w \rangle$) is invariant under $\Phi$. 
\end{proof}

\section{The edgewise nonseparating free splitting graph}\label{sec:ens}

\emph{In this section, we introduce the edgewise nonseparating free splitting graph and show that all its graph automorphisms come from the action of $\Out(F_N)$.}
\\
\\
\indent We let $M_{N}:=\sharp_{i=1}^N (S^1\times S^2)$ be the connected sum of $N$ copies of $S^1 \times S^2$, and we identify once and for all the fundamental group of $M_N$ with the free group $F_N$. 
We recall that every embedded sphere in $M_N$ which does not bound a ball determines a one-edge free splitting of $F_N$ in the following way. The fundamental group of a sphere is trivial, so an application of van Kampen's theorem shows that an embedded sphere determines a splitting of the fundamental group of $M_N$ (which has been identified with $F_N$) over the trivial group. The fact that the sphere does not bound a ball ensures that the splitting defined in this way is nontrivial. Conversely, every one-edge free splitting of $F_N$ can be represented by the isotopy class of an essential embedded sphere in $M_N$ (this is described in the appendix of \cite{Hat} in the case of simple sphere systems but the proof extends to all free splittings. See also \cite{Sta1,AS}). More generally, every collection $\Sigma$ of essential (i.e.\ which do not bound a ball), pairwise disjoint, pairwise non-isotopic spheres determines a free splitting of $F_N$ whose one-edge collapses are precisely the one-edge free splittings determined by each of the spheres in $\Sigma$. Such a collection is called a \emph{sphere system}. A sphere system $\Sigma$ (or the corresponding free splitting $S$ of $F_N$) is \emph{simple} if all the components of $M_N-\Sigma$ have trivial fundamental group (equivalently, the $F_N$-action on $S$ is free).
In this section, we will abuse notation in places by blurring the distinction between a sphere system and its induced free splitting as well as the distinction between an edge in such a splitting and its associated sphere in $M_N$.

\begin{de}[Free splitting graph]
 The \emph{free splitting graph} $\mathrm{FS}$ is the graph whose vertices are the (homeomorphism classes of) one-edge free splittings of $F_N$, two vertices being joined by an edge whenever they are compatible. 
\end{de}

\begin{de}[Nonseparating free splitting graph]
The \emph{nonseparating free splitting graph} $\NS$ is the graph whose vertices are the (homeomorphism classes of) loop-edge free splittings of $F_N$, two vertices being joined by an edge whenever they are compatible.
\end{de}

Due to the correspondence between spheres and one-edge splittings, $\NS$ can alternatively be thought as the graph whose vertices are nonseparating spheres in $M_N$, with edges given by disjointness. The following theorem, established by Pandit in \cite{Pan}, heavily relies on previous work of Bridson and Vogtmann \cite{BV2} giving a similar rigidity statement for the spine of reduced Outer space.

\begin{theo}[Pandit \cite{Pan}]\label{pandit}
For every $N\ge 3$, the natural map $\Out(F_N)\to\Aut(\NS)$ is an isomorphism.
\end{theo}

\begin{proof}[Sketch proof]
A system of nonseparating spheres $\Sigma$ (equivalently, a clique in $\NS$) determines a simplex in the spine of reduced Outer space $K_N$ if and only if it is simple.
 The graph corresponding to a sphere system has finitely many blowups if and only if it is simple or there is a leaf with $\mathbb{Z}$ as the vertex stabilizer. However, in the latter case this leaf edge (equivalently, its corresponding sphere) is separating. Therefore a system $\Sigma$ of nonseparating spheres is simple if and only if the link of the clique corresponding to $\Sigma$  is finite in $\NS$. Furthermore the simplex determined by $\Sigma$ is a face of the simplex determined by $\Sigma'$ in $K_N$ if and only if $\Sigma \subset \Sigma'$. These conditions are preserved under automorphisms of $\NS$, so we have an induced map $\Phi: \Aut(\NS) \to \Aut(K_N)$ which is equivariant under the $\Out(F_N)$ action. This induced map is also injective: if $\sigma$ and $\sigma'$ are distinct nonseparating splittings we can find a simple sphere system $\Sigma$ containing one but not the other. It follows that if an automorphism $\alpha \in \Aut(\NS)$ induces the identity on the spine it fixes $\Sigma$ setwise and cannot send $\sigma$ to $\sigma'$. Hence $\alpha$ is also the identity on $\NS$. We have the following commutative diagram:
\[
  \begin{tikzcd}
 \Out(F_N) \arrow[bend left=30]{rr}{\Psi}   
     \arrow{r} & \Aut(\NS) \arrow{r}{\Phi} & \aut(K_N).
      \end{tikzcd}
\]
A theorem of Bridson and Vogtmann \cite{BV2} states that the natural map $\Psi: \Out(F_N) \to \Aut(K_N)$ is an isomorphism, in particular $\Phi$ is also surjective and hence is an isomorphism.
\end{proof}

\begin{de}[Edgewise nonseparating free splitting graph]
The \emph{edgewise nonseparating free splitting graph} $\ens$ is the graph whose vertices are the (homeomorphism classes of) loop-edge free splittings of $F_N$, two vertices being joined by an edge whenever they are compatible and have a two-petal rose refinement (equivalently, the complement of the union of the two corresponding spheres in $M_N$ is connected).
\end{de}

Informally, we define $\ens$ by throwing out all of the edges in $\NS$ that are given by a pair of disjoint nonseparating spheres whose union separates.  The dual graph given by such a pair of spheres is a loop with two edges. Such a pair of spheres are then of distance 2 in $\ens$ (we will see a refinement of this statement in the claim within the proof of Theorem~\ref{ens-automorphisms}).

\begin{theo}\label{ens-automorphisms}
For every $N\ge 3$, the natural map $\theta:\Out(F_N)\to\Aut(\ens)$ is an isomorphism.
\end{theo}

\begin{proof}
As no free splitting of $F_N$ is invariant by every element of $\Out(F_N)$, the map $\theta$ is injective. We now focus on proving that $\theta$ is onto.

Let $\Psi\in\Aut(\ens)$. In view of Theorem~\ref{pandit}, it is enough to show that $\Psi$ can be extended to a simplicial automorphism of $\NS$. In other words, we wish to show that if $S$ and $S'$ are two distinct compatible splittings whose common refinement is a two-edge loop, then the same is true for $\Psi(S)$ and $\Psi(S')$. It is enough to prove the following claim.
\\
\\
\textbf{Claim:} Let $S$ and $S'$ be two splittings such that $d_{\ens}(S,S')>1$. The following are equivalent.
\begin{itemize}
\item We have $d_{\NS}(S,S')=1$, in other words $S$ and $S'$ are compatible, and denoting by $U$ their common refinement, the graph $U/F_N$ is a loop.
\item The intersection $\lk(S)\cap\lk(S')$ in $\ens$ contains a clique with finite link of size $3N-5$, but no clique of size $3N-4$. Furthermore, $\lk(S)\cap\lk(S')$ is not a cone over a point. 
\end{itemize}

\begin{figure}
\centering
\input{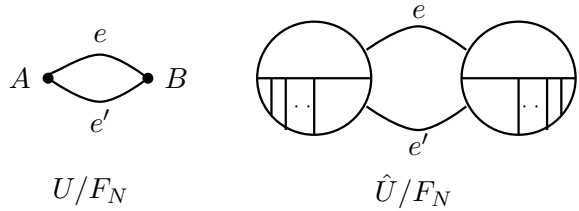}
\caption{A maximal clique in the common link of two compatible splittings whose common refinement is a two-edge loop.}
\label{fig:maximal-clique}
\end{figure}

 We now prove the above claim. First assume that $d_{\NS}(S,S')=1$; in other words, the splittings $S$ and $S'$ are compatible, and denoting by $U$ their common refinement, the graph $U/F_N$ is a two-edge loop. One can then blow up each of the vertex groups of the loop to get the splitting $\hat{U}$ depicted in Figure~\ref{fig:maximal-clique}. The graph $\hat{U}/F_N$ is a trivalent graph whose fundamental group has rank $N$, so it contains $3N-3$ edges. Given any two edges $e_1$ and $e_2$ that are not equal to $e$ or $e'$, the graph obtained from $\hat{U}/F_N$ by collapsing all edges but $e_1$ and $e_2$ is a two-petalled rose. This shows that $\lk(S)\cap\lk(S')$ contains a clique of size $3N-5$. In addition, the splitting obtained from $\hat{U}$ by collapsing the orbits of $e$ and $e'$ is simple, so this clique has finite link. Notice also that $\lk(S)\cap\lk(S')$ cannot contain a clique of size $3N-4$, as adding $e$ and $e'$ to this clique in $\NS$ would yield a free splitting of $F_N$ with $3N-2$ orbits of edges, which is impossible. As there are incompatible blowups at each of the two vertices of $U/F_N$, we see that  $\lk(S)\cap\lk(S')$ is not a cone over a point.

Conversely, let us assume that $\lk(S)\cap\lk(S')$ in $\ens$ contains a clique with finite link of size $3N-5$, no clique of size $3N-4$, and is not a cone over a point. Assume that $d_{\NS}(S,S')>1$, i.e.\ $S$ and $S'$ are not compatible. Then $S$ and $S'$ lie in a complementary region of a simple sphere system (corresponding to the clique with finite link of size $3N-5$). Such a complementary region is a 3-sphere with finitely many open balls removed. Any two spheres in such a region are disjoint or, up to isotopy, intersect in a single circle. This is ruled out by a case-by-case analysis in Lemma~\ref{cliques}, below.
\end{proof}

\subsection{Spheres intersecting in a single circle}

Suppose $S$ and $S'$ are two spheres that intersect in a single essential circle. The circle separates each sphere into two discs, and the regular neighbourhood of the union of $S$ and $S'$ is a $3$-sphere with four boundary components, each of which is a $2$-sphere isotopic to the union of one half of $S$ and one half of $S'$. We refer to these four spheres as the \emph{boundary spheres} of $S$ and $S'$. Each boundary sphere is essential, as otherwise the circle of intersection between $S$ and $S'$ would not be essential. It might happen that two of these spheres are isotopic.

The four boundary spheres of $S$ and $S'$ determine a free splitting $U$ of $F_N$, which we call the \emph{boundary splitting} of $S$ and $S'$. The vertices of the quotient graph $U/F_N$ correspond to complementary regions in $M_N$ of the union of the boundary spheres (in the case where two boundary spheres are isotopic, we ignore the redundant valence $2$ vertex associated to the region bounded by the two spheres). One of these complementary regions is precisely the regular neighbourhood of $S$ and $S'$, which has trivial fundamental group. We refer to the corresponding vertex of $U/F_N$ as the \emph{central vertex} of $U/F_N$; it has valence four and every lift of this vertex in $U$ has trivial $F_N$-stabilizer. 

We claim that if $N \geq 3$, then at most two of the boundary spheres are isotopic and form a loop at the central vertex. Otherwise, the quotient graph of groups $U/F_N$ would be a $2$-petal rose. As the central vertex has trivial vertex group, this implies that $N=2$. In the case where where exactly two boundary spheres are isotopic, then $U/F_N$ has exactly three edges, one of which is a loop-edge. In the case where the boundary spheres are pairwise non-isotopic, the quotient graph $U/F_N$ has exactly four edges. 

In all cases, the spheres $S$ and $S'$ correspond to distinct blowups of the four half-edges at the central vertex (combinatorially these are obtained by a partition of the four half-edges at the central vertex into two subsets of two half-edges). In order for both $S$ and $S'$ to be nonseparating, at least two half-edges are adjacent to the same connected component of $M_N$ with these three or four spheres removed. The boundary splitting can have one of six types, depicted in Figure~\ref{fig:boundary-spheres}. 
\begin{figure}
\centering
\input{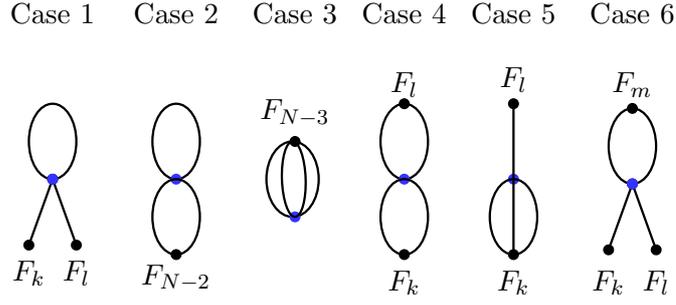}
\caption{The six possibilities for (the quotient graph of groups of) the boundary splitting. The central vertex is depicted in blue, and has trivial stabilizer.}
\label{fig:boundary-spheres}
\end{figure}

\begin{enumerate}[1.]
\item One loop, two non-central vertices with fundamental groups $F_k$ and $F_l$ respectively, with $k+l=N-1$.
\item One loop, one non-central vertex with fundamental group $F_{N-2}$.
\item No loop, one non-central vertex with fundamental group $F_{N-3}$.
\item No loop, two non-central vertices each adjacent to two of the boundary spheres with fundamental groups $F_k$ and $F_l$ with $k+l=N-2$.
\item No loop, two non-central vertices, one of which is adjacent to one of the boundary spheres, the other of which is adjacent to three of the boundary spheres, where the fundamental groups are $F_k$ and $F_l$ with $k+l=N-2$ (here possibly $k=0$).
\item No loop, three non-central vertices with fundamental groups $F_k$, $F_l$ and $F_m$ with $k+l+m=N-1$.
\end{enumerate}

We are now in a position to study maximal cliques in the joint links of $S$ and $S'$ in $\ens$. This completes the proof of Theorem~\ref{ens-automorphisms}.

\begin{lemma}\label{cliques}
If $S$ and $S'$ are two nonseparating spheres in $M_N$ which intersect in a single circle (when in normal form) then either
\begin{itemize}
 \item  $\lk(S)\cap\lk(S')$ does not contain a clique of size $3N-5$,
 \item  $\lk(S)\cap\lk(S')$ contains a clique of size $3N-4$, or
 \item  $\lk(S)\cap\lk(S')$ is a cone over a point.
\end{itemize}

\end{lemma}

\begin{proof}
We study the joint link $\lk(S)\cap\lk(S')$ in $\ens$ on a case-by-case basis. In each case, we will see that one of the above conditions is satisfied. Note that every sphere in $\lk(S)\cap\lk(S')$ is either a boundary sphere or disjoint from the boundary spheres, so that every clique in $\lk(S)\cap\lk(S')$ can be refined to a blowup of the boundary splitting.

Let $\Sigma$ be a maximal clique in $\lk(S)\cap\lk(S')$ in $\ens$. If there exist two distinct boundary spheres $S_1$ and $S_2$ which are both not contained in $\Sigma$, then $\lk(S)\cap\lk(S')$ does not contain a clique of size $3N-5$. This is because any maximal clique in $\lk(S)\cap\lk(S')$ can be extended by these two boundary spheres and either $S$ or $S'$ to form a clique in $\mathrm{FS}$, therefore contains at most $3N-6$ vertices. This applies to Cases~1 and~6 as there are two distinct separating boundary spheres in these splittings. It also applies to Case~4, as each pair of boundary spheres with the same endpoints are non-adjacent in $\ens$, so that at most two can be contained in a maximal clique in $\ens$. Furthermore, we can also apply this to Case~5: by examining the blowup given by $S$ one sees that only two of the three nonseparating edges are adjacent to $S$ in $\ens$. Therefore this sphere and the separating boundary sphere are not contained in $\lk(S)\cap\lk(S')$. 
 
In Case~2, the loop edge is adjacent to both $S$ and $S'$ in $\ens$, as well as both of the other boundary spheres and every blowup of the non-central vertex. Therefore $\lk(S)\cap\lk(S')$ is a cone over the splitting corresponding to this loop edge.

\begin{figure}
\centering
\input{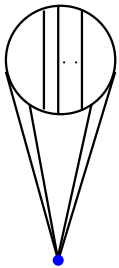}
\caption{A blowup of the boundary splitting with $3N-4$ orbits of edges in Case~3.}
\label{fig:blowup}
\end{figure}

In case 3, Figure~\ref{fig:blowup} represents a blowup of the boundary splitting with $3N-4$ orbits of edges, such that every one-edge collapse is in the common link of $S$ and $S'$ in $\ens$. 
This gives a clique in $\lk(S)\cap\lk(S')$ of size $3N-4$.
\end{proof}

\section{Direct products acting on hyperbolic spaces}\label{sec:direct-product-vs-hyp}

\emph{As explained in the introduction, a key feature used in the proof of our main theorem is that stabilizers of one-edge nonseparating free splittings contain a normal subgroup which is a direct product of two free groups. In this section, we describe how such normal subgroups restrict actions on hyperbolic spaces. }
\\ 
\\
\indent Given an isometric action of a group $H$ on a metric space $X$, we say that $H$ \emph{has bounded orbits} in $X$ if for every $x\in X$, the diameter of the orbit $H\cdot x$ is finite. When $X$ is Gromov hyperbolic, we use $\partial_\infty X$ to denote the Gromov boundary and $\partial_H X$ to denote the \emph{limit set} of $H$ in $\partial_\infty X$, i.e.\ the space of all accumulation points of $H\cdot x$ in $\partial_\infty X$, where $x\in X$ is any point. In particular, if $H$ has bounded orbits then $\partial_H X$ is empty and if $\Phi$ is a loxodromic isometry then $\partial_{\langle \Phi \rangle} X$ is a two point set consisting of the attracting and repelling points of $\Phi$. The following theorem of Gromov \cite{Gro} (see also \cite[Proposition~3.1]{CCMT}) classifies group actions on hyperbolic metric spaces. 
Note that the action is not required to be proper.

\begin{theo}[Gromov]\label{gromov} 
Let $X$ be a geodesic Gromov hyperbolic metric space, and let $H$ be a group acting by isometries on $X$. Then either
\begin{itemize}
\item $H$ contains two loxodromic isometries of $X$ that generate a free subgroup of $H$, or 
\item the limit set $\partial_H X$ contains a finite nonempty $H$-invariant subset, or else
\item $H$ has bounded orbits in $X$.  
\end{itemize} 
\end{theo}

If $K$ is a subgroup of $H$, then the centralizer of $K$ in $H$ fixes the limit set $\partial_K X$ pointwise. The goal of this section is to combine this observation with Gromov's theorem to prove the following:

\begin{prop}\label{product-vs-hyp}
Let $X$ be a geodesic Gromov hyperbolic metric space, and let $H$ be a group acting by isometries on $X$. Assume that $H$ contains a normal subgroup $K$ which is isomorphic to a direct product $K=\prod_{i=1}^kK_i$.  
\\ If some $K_j$ contains a loxodromic element then $\prod_{i \neq j} K_i$ has a finite orbit in $\partial_\infty X$.
\\ If no $K_j$ contains a loxodromic element, then either $K$ has a finite orbit in $\partial_\infty X$ or $H$ has bounded orbits in $X$.
\end{prop}

Before the proof, we give a brief lemma that describes the case when $K$ has bounded orbits.

\begin{lemma}\label{bounded-normal}
Let $X$ be a geodesic Gromov hyperbolic metric space, and let $H$ be a group acting by isometries on $X$. Assume that $H$ contains a normal subgroup $K$ that has bounded orbits in $X$.
\\ Then either $K$ fixes a point in $\partial_\infty X$ or $H$ has bounded orbits in $X$. 
\end{lemma}

\begin{proof}
As $K$ has bounded orbits in $X$, we can find $M>0$ such that $$Y:=\{x\in X|\diam(K\cdot x)\le M\}$$ is nonempty. Since $K$ is normal in $H$, the set $Y$ is $H$-invariant: this follows from the fact that for all $x\in X$ and all $h\in H$, we have $\diam(K\cdot h x)=\diam(h K\cdot x)=\diam(K\cdot x)$.
If $Y$ has an accumulation point in $\partial_\infty X$, then this is a fixed point of $K$ in $\partial_\infty X$. Otherwise, $Y$ is a $H$-invariant subset of $X$ with no accumulation point in $\partial_\infty X$, so in particular $\partial_H X=\emptyset$. By Theorem~\ref{gromov}, this implies that $H$ has bounded orbits in $X$.  
\end{proof}

\begin{proof}[Proof of Proposition~\ref{product-vs-hyp}] 
If $K_j$ contains a loxodromic isometry $\Phi$ of $X$, then $\prod_{i\neq j}K_i$ commutes with $\Phi$ and therefore fixes the two-point set $\partial_{\langle\Phi\rangle}X$ consisting of the attracting and repelling points of $\Phi$ in the boundary. We may therefore assume that no subgroup $K_i$ contains a loxodromic isometry.

If there exists $j\in\{1,\dots,k\}$ such that $K_j$ has unbounded orbits in $X$, 
 Theorem~\ref{gromov} implies that $K_j$ has a finite invariant set in $\partial_\infty X$, which is also fixed by the subgroup $\prod_{i\neq j}K_i$ which commutes with $K_j$. Hence $K$ has a finite orbit in $\partial_\infty X$.

In view of Theorem~\ref{gromov}, we are thus left with the case where all subgroups $K_i$ have bounded orbits in $X$,  
in which case it is not hard to see that $K$ itself has bounded orbits in $X$. As $K$ is normal in $H$, Lemma~\ref{bounded-normal} implies that either the whole group $K$ has a fixed point in $\partial_\infty X$ or $H$ has bounded orbits in $X$. 
\end{proof}

When at least two of the subgroups $K_i$ contain a loxodromic isometry we can say a bit more, namely that the whole group $H$ has a finite orbit in $\partial_\infty X$.

\begin{lemma}\label{loxo-loxo}
Let $X$ be a geodesic Gromov hyperbolic metric space, and let $H$ be a group acting by isometries on $X$.
\\ Assume that $H$ contains a normal subgroup $K$ which is isomorphic to a direct product $K=\prod_{i=1}^kK_i$, and that there exists $j,l\in\{1,\dots,k\}$ with $j\neq l$ such that both $K_j$ and $K_l$ contain a loxodromic isometry of $X$.  
\\ Then $H$ has a finite orbit in $\partial_\infty X$. 
\end{lemma}

\begin{proof}
Let $\Phi_j\in K_j$ be a loxodromic isometry of $X$. Then for every $i\neq j$, the group $K_i$ centralizes $\Phi_j$, hence  fixes the two-point set $\partial_{\langle \Phi_j\rangle}X$. If $\Phi_l$ is a loxodromic isometry in $K_l$, then as $\Phi_l$ and $\Phi_j$ commute we have $\partial_{\langle\Phi_l\rangle}X=\partial_{\langle\Phi_j\rangle}X$. Therefore $K_j$ also fixes the pair $\partial_{\langle\Phi_j\rangle}X$, and this is the only $K$-invariant pair in $\partial_\infty X$. As $K$ is normal in $H$, we deduce that this pair of points is $H$-invariant. 
\end{proof}

\section{Stabilizers of relatively arational trees}\label{sec:arat}

\emph{When a direct product of subgroups of $\Out(F_N,\calf)$ acts on the relative free factor graph $\mathrm{FF}:=\FF(F_N,\calf)$, the previous section forces its action to be elementary: either it has bounded orbits, or one factor has a finite orbit in the boundary. This suggests that one needs to understand stabilizers of points in $\partial_\infty\mathrm{FF}$, which up to finite index are stabilizers of relatively arational trees. Understanding these stabilizers is the goal of the present section.}
\\
\\
\indent Our main result in this section will be the following proposition.

\begin{prop}\label{prop:stab-arat-out}
Let $K\subseteq\ia$ be a subgroup contained in the isometric stabilizer of a relatively arational tree.
\\ Then $K$ virtually centralizes a subgroup of $\Out(F_N)$ isomorphic to $\mathbb{Z}^3$.
\end{prop}

Later, in Section~\ref{sec:later}, we will also describe stabilizers of arational trees in the subgroups of $\out(F_N)$ that appear in our main theorem.
We will first give some background about the structure of relatively arational trees before giving the proof of Proposition~\ref{prop:stab-arat-out} in Section~\ref{s:proof_stab_rel_aration}.

\subsection{Transverse coverings of arational trees and their skeletons}

Let $T$ be a minimal $F_N$-tree. Recall from \cite[Definition~4.7]{Gui} that a \emph{transverse family} in $T$ is an $F_N$-invariant collection $\caly$ of nondegenerate subtrees of $T$ such that any two distinct subtrees in $\caly$ intersect in at most one point. It is a \emph{transverse covering} if in addition, every subtree in $\caly$ is closed, and every segment in $T$ is covered by finitely many subtrees in $\caly$.

Every transverse covering $\caly$ of $T$ has an associated \emph{skeleton} $S$, as defined in \cite[Definition~4.8]{Gui}.  This is the bipartite tree with one vertex $v_Y$ for every subtree $Y\in\caly$ and one vertex $v_x$ for every point $x\in T$ which is contained in at least two different subtrees of $\caly$. There is an edge joining $v_x$ to $v_Y$ whenever $x\in Y$. By \cite[Lemma~4.9]{Gui}, the tree $S$ is minimal as an $F_N$-tree. We will usually denote by $G_Y$ the stabilizer of the vertex $v_Y$. 

\begin{lemma}\label{lemma:minimal}
Let $\caly$ be a transverse family in a very small $F_N$-tree with dense orbits.  If $Y \in \caly$ then the action of $\Stab(Y)$ on $Y$ has dense orbits. In addition, either the action of $\stab(Y)$ on $Y$ is minimal or the skeleton of $\caly$ has an edge with trivial stabilizer.
\end{lemma}

\begin{proof}
Suppose that the action of $\Stab(Y)$ on $Y$ does not have dense orbits. Then there exist $x,x_0 \in Y$ such that \[ \epsilon := \inf\{d(gx,x_0): g \in \Stab(Y) \}  >0. \]
Pick $x'$ in the $\Stab(Y)$-orbit of $x$ and suppose that $d(x',x_0)=a\epsilon$ for some $a > 1$ (notice that $[x_0,x']\subseteq Y$ and we can choose $x'$ so that $a$ is arbitrarily close to 1). Recall that a \emph{direction} at a point $y \in T$ is a component $d_y$ of $T - \{y\}$. Let $X$ be the set of branch directions containing $x_0$ based at points in the interior of the segment $[x',x_0]$ (i.e. $X$ is the set of branch directions in $[x',x_0]$ pointing towards $x_0$). By Lemma~4.2 of \cite{LL} (see also \cite{GabL}), arc stabilizers in $T$ are trivial and the number $B$ of $F_N$-orbits of branch directions is finite. Furthermore, the fact that the $F_N$-action on $T$ has dense orbits implies that the branch points are dense in $[x',x_0]$. It follows that for any $C>B$, there exist two directions $d,d' \in X$ in the same $F_N$-orbit based at points at least $a\epsilon/C$ apart. Let $d=d_y$ and let $g \in F_N$ such that $gd=d'$. After possibly swapping the directions (and $g$ with $g^{-1}$) we may assume that $gy$ is closer to $x_0$ than $y$ (by at least $a\epsilon/C$). As $g$ sends the direction at $y$ containing $x_0$ to the direction at $gy$ containing $x_0$ and $[x_0,x']\subseteq Y$, we see that $gY \cap Y$ is non-degenerate and $g \in \stab(Y)$. If $x'$ is chosen so that $a<C/(C-1)$, then \[ d(gx',x_0)\leq a\epsilon - {a\epsilon}/{C}< \epsilon, \]   which is a contradiction. Hence $\Stab(Y)$ acts on $Y$ with dense orbits. 

For the second point, suppose that the minimal subtree $Y'$ of $\stab(Y)$ is not equal to $Y$. Let $x$ be a point in $Y-Y'$. As $x$ is in the closure of $Y'$ there exists a unique direction $d$ at $x$ which intersects $Y$. Therefore $x$ lies in more than one element of $\caly$ and determines a vertex in the skeleton $S$. The stabilizer of the edge between $x$ and $Y$ also fixes the direction $d$, and such stabilizers are trivial in very small $F_N$-trees with dense orbits (or indeed any tree with trivial arc stabilizers). 
\end{proof}

A tree is \emph{mixing} if given any two segments $I,J\subseteq T$, there exists a finite set $\{g_1,\dots,g_k\}$ of elements of $F_N$ such that $J\subseteq g_1I\cup\dots\cup g_kI$. Any mixing tree has dense $F_N$-orbits. The mixing condition implies that any transverse family $\caly$ of closed subtrees is a transverse covering, and $\caly$ has only one orbit under $F_N$. Relatively arational trees are mixing by \cite[Lemma~4.9]{Hor0}, and the skeleton given by a transverse covering of an arational tree satisfies the following properties:

\begin{lemma}\label{l:claim-1}
Let $\calf$ be a free factor system of $F_N$. Let $\caly$ be a transverse covering of an arational $(F_N,\calf)$-tree $T$ and let $S$ be the skeleton of $\caly$.
\begin{itemize}
 \item There is exactly one $F_N$-orbit of vertices of the form $v_Y$ in $S$ (in other words, $F_N$ acts transitively on $\caly$).  
 \item The stabilizer of every edge of $S$ is nontrivial, and cyclic edge stabilizers are peripheral.
 \item The stabilizer of every vertex of the form $v_x$ is an element of $\calf$.
\end{itemize}
\end{lemma}

\begin{proof}
As we discussed above, the first assertion follows from the fact that $T$ is mixing. 

We will now prove the second assertion of the lemma. We first observe that every subgroup $A$ in $\calf$ is elliptic in $S$. Indeed, the group $A$ is elliptic in $T$ and fixes a unique point $x$. If $x$ is contained in a single subtree $Y\in\caly$, then $Y$ is $A$-invariant so $A$ fixes $v_Y$ in $S$. Otherwise $x$ is contained in at least two distinct subtrees in $\caly$ and $A$ fixes the point $v_x$ in $S$. 

Now, suppose that an edge stabilizer $G_e$ is trivial or cyclic and nonperipheral. Collapsing all other edge orbits in $S$, we obtain a decomposition of $T$ as a graph of actions where each vertex group $G_v$ is either a free factor of $(F_N,\calf)$, or more generally a \emph{proper $\mathcal{Z}$-factor of $(F_N,\calf)$}  as defined in Section~11.4 of \cite{GH1} (i.e.\ a nonperipheral subgroup that arises as a point stabilizer in a splitting of $F_N$ relative to $\calf$ whose edge groups are either trivial or cyclic and nonperipheral).  If $G_v$ is a free factor then the action of $G_v$ on its minimal subtree in $T$ is simplicial as $T$ is arational as an $(F_N,\calf)$-tree, and more generally \cite[Proposition~11.5]{GH1} tells us the same thing is true if $G_v$ is a proper $\mathcal{Z}$-factor. Therefore the whole action of $F_N$ on $T$ is simplicial, which is a contradiction.

We now prove the third assertion of the lemma. The stabilizer of every vertex of the form $v_x$ is a point stabilizer $G_x$ in $T$ so is either trivial, an element of $\calf$, or cyclic and nonperipheral (this follows from \cite[Lemma~4.6]{Hor0} -- the cyclic, nonperipheral stabilizers come from \emph{arational surface trees}). However, $G_x$ contains an edge stabilizer in $S$, so by the above work $G_x$ has to be an element of $\calf$.
\end{proof}

\subsection{Canonical piecewise-$F_N$ coverings}

Given a subgroup $K\subseteq\Out(F_N)$, we denote by $\tilde{K}$ the full preimage of $K$ in $\Aut(F_N)$. Now let $K\subseteq\Out(F_N)$ be a subgroup contained in the isometric stabilizer of $T$: this means that for every $\alpha\in\tilde{K}$, there exists an isometry $I_\alpha$ of $T$ which is \emph{$\alpha$-equivariant} in the sense that for every $g\in F_N$, one has $I_\alpha(gx)=\alpha(g)I_\alpha(x)$ (and such a map $I_\alpha$ is actually unique, see e.g.\ \cite[Corollary~3.7]{KL}). Assume that the transverse covering $\caly$ is $K$-invariant. We say that $\caly$ is \emph{$K$-piecewise-$F_N$} if there exists a map $g:\tilde{K}\times\caly\to F_N$ such that for every $\alpha\in\tilde{K}$ and every $Y\in\caly$, the automorphism $\alpha$ induces the same action on $Y$ as $g(\alpha,Y)$.  Using the fact that $T$ has trivial arc stabilizers and subtrees in $\caly$ are nondegenerate, we get that such a map $g$ is unique. 

Given an outer automorphism $\Phi$ in the isometric stabilizer of an $F_N$-tree $T$, we say that $\Phi$ \emph{preserves all orbits of branch directions} in $T$ if for some (equivalently any) representative $\alpha$ of $\Phi$ in $\Out(F_N)$, the isometry $I_{\alpha}$ sends every branch direction in $T$ to a branch direction in the same orbit. More generally, we say that a subgroup $K$ of the isometric stabilizer of $T$ \emph{preserves all orbits of branch directions} in $T$ if every element in $K$ does. Since by \cite{GabL}, there is a bound on the number of branch directions in a very small $F_N$-tree $T$, every subgroup of the isometric stabilizer of $T$ has a finite-index subgroup that preserves all orbits of branch directions.

Recall that $G \subseteq F_N$ is a \emph{fixed subgroup} of $K \subseteq \out(F_N)$ if every element of $K$ has a representative in $\Aut(F_N)$
acting as the identity on $G$. If $G$ is noncyclic then every outer automorphism has a unique representative fixing $G$, so that $G$ determines a lift $\tilde K_G$ of $K$ to $\Aut(F_N)$. By \cite{DV}, the maximal fixed subgroup of every collection of outer automorphisms of $F_N$ is finitely generated (of rank at most $N$).

There is a natural partial ordering on the collection of all transverse coverings of a given $F_N$-tree $T$, by letting $\caly\le\caly'$ whenever $\caly$ refines $\caly'$ (in other words every subtree in $\caly$ is contained in a subtree in $\caly'$). Any pair of coverings $\mathcal{Y}$ and $\mathcal{Y}'$ have a maximal common refinement given by the nondegenerate intersections of their elements. Any finite collection of transverse coverings has a maximal common refinement in a similar fashion. 

The following theorem is due to Guirardel and Levitt; we include a proof, which we learned from Vincent Guirardel, only for completeness. 
 
\begin{theorem}[Guirardel--Levitt \cite{GL}]\label{th:gl}
Let $\calf$ be a free factor system of $F_N$, and let $T$ be an arational $(F_N,\calf)$-tree. Let $K \subseteq \out(F_N)$ be a subgroup of the isometric stabilizer of $T$, and let $K^0$ be the finite-index subgroup of $K$ made of all elements that preserve all orbits of branch directions in $T$.
\\ Then there exists a unique maximal $K^0$-piecewise-$F_N$ transverse covering $\caly$ of $T$. In addition, the stabilizer $G_Y$ of any subtree $Y\in\caly$ is (up to conjugation) the unique maximal noncyclic nonperipheral fixed subgroup of $K^0$ (in particular $G_Y$ is finitely generated). 
\end{theorem}

We call $\caly$ the \emph{$K$-canonical piecewise-$F_N$ transverse covering of $T$}. 

\begin{proof}
We first deal with the case where $K^0$ is a cyclic group, generated by a single outer automorphism $\Phi$. Let $\alpha\in\Aut(F_N)$ be a representative of $\Phi$, and let $I_{\alpha}$ be the unique $\alpha$-equivariant isometry of $T$. For every $g\in F_N$, we let $Y_g:=\{x\in T| I_{\alpha}(x)=gx\}$. Each $Y_g$ is a subtree, and since $\Phi$ preserves all orbits of branch directions in $T$, at least one of the subtrees $Y_g$ is nondegenerate. Since $T$ has trivial arc stabilizers, if $Y_g \cap Y_h$ is nondegenerate then $g=h$, so the family $\caly$ made of all nondegenerate subtrees of the form $Y_g$ is a transverse family in $T$. As $T$ is mixing and all subtrees in $\caly$ are closed, $\caly$ is a transverse covering, and by construction it is the unique maximal $K^0$-piecewise-$F_N$ transverse covering of $T$.

More generally, if $K^0$ is finitely generated, then construct coverings $\mathcal{Y}_1, \ldots, \mathcal{Y}_k$ as above for a generating set $\Phi_1, \ldots, \Phi_k$ of $K^0$ and let $\mathcal{Y}$ be the maximal common refinement of the $\caly_i$. By construction, $\caly$ is the unique maximal $K^0$-piecewise-$F_N$ transverse covering of $T$. Let $Y$ be a subtree in the family $\caly$, and let $G_Y$ be its stabilizer. We will now prove that $G_Y$ is (up to conjugation) the unique maximal noncyclic nonperipheral fixed subgroup of $K^0$. 

By Lemma~\ref{l:claim-1}, the skeleton of $\caly$ does not contain any edge with trivial stabilizer. Lemma~\ref{lemma:minimal} therefore ensures that $Y$ is the minimal invariant subtree of its stabilizer $G_Y$.
As peripheral subgroups are elliptic in $T$, this tells us that $G_Y$ is nonperipheral. As a cyclic group cannot act on a nondegenerate tree with dense orbits,  $G_Y$ is noncyclic by Lemma~\ref{lemma:minimal}.  

We will first show that $G_Y$ is a fixed subgroup of $K^0$. Every element of $K^0$ has a unique representative that acts as the identity on $Y$. To see this, if $\alpha \in \tilde K^0$, then there exists $g\in F_N$ such that for every $x\in Y$, one has $I_\alpha(x)=gx$. Hence $\ad_g^{-1}\alpha$ acts as the identity on $Y$. This representative is unique as $Y$ is nondegenerate and $T$ has trivial arc stabilizers. 
Let $\tilde{K}_Y$ be the lift of $K^0$ to $\Aut(F_N)$ made of all such automorphisms. For every $g\in G_Y$, every $y\in Y$, and every $\alpha\in \tilde{K}_Y$, one has $gy=I_{\alpha}(gy)=\alpha(g)I_\alpha(y)=\alpha(g)y$. As $T$ has trivial arc stabilizers, this implies that $\alpha(g)=g$ and $\alpha_{|G_Y}=\mathrm{id}$. 

We will now prove the maximality of $G_Y$. Let $A$ be a noncyclic nonperipheral subgroup of $F_N$ such that every element of $K^0$ has a representative $\alpha\in\Aut(F_N)$ such that $\alpha_{|A}=\mathrm{id}$. Notice that the $A$-minimal subtree $T_A$ is nontrivial because $T$ is relatively arational (recall that the only nonperipheral point stabilizers in $T$ are cyclic). Let $a\in A$ be an element that acts hyperbolically on $T$. Then $I_{\alpha}$ preserves the axis of $a$, so acts on it by translation. Given an element $b\in A$ acting hyperbolically on $T$ such that $\langle a,b\rangle$ is noncyclic, the intersection of the axes of $a$ and $b$ in $T$ is compact (possibly empty). The isometry $I_{\alpha}$ also preserves the axis of $b$, and therefore it fixes a point on the axis of $a$ (namely, the projection of the axis of $b$ to the axis of $a$ if these do not intersect, or otherwise the midpoint of their intersection). Therefore $I_{\alpha}$ acts as the identity on the axis of every hyperbolic element of $A$. This implies that $I_{\alpha}$ acts as the identity on the $A$-minimal subtree $T_A\subseteq T$ and its closure $\overline{T}_A$.  Notice that the family $\{g\overline{T}_A\}_{g\in F_N}$ is a transverse family in $T$ (indeed $I_\alpha$ acts like identity on $\overline{T}_A$ and like $\alpha(g)$ on $g\overline{T}_A$ and $T$ has trivial arc stabilizers). As $T$ is mixing, the family $\{g\overline{T}_A\}_{g\in F_N}$ is a transverse covering. The maximality property of the covering $\caly$ implies that $\overline{T}_A\subseteq Y$ for some $Y \in \caly$. If $a\in A$, then $aY\cap Y$ contains $T_A$. Since $\caly$ is a transverse family, this implies that $aY=Y$. This proves that $A$ is a subgroup of $G_Y$. This finishes the proof of the theorem when $K^0$ is finitely generated. 

We now deal with the general case. Let $(K_i)_{i\in\mathbb{N}}$ be an increasing sequence of finitely generated subgroups of $K^0$ such that $K^0=\bigcup_{i\in\mathbb{N}}K_i$. For every $i\in\mathbb{N}$, let $\caly_i$ be the $K_i$-canonical piecewise-$F_N$ transverse covering of $T$. We will prove that the coverings $\caly_i$ stabilize for $i$ sufficiently large. Let $Y_i$ be a subtree in $\caly_i$, and let $G_i$ be its stabilizer. For every $i\in\mathbb{N}$, we have $G_{i+1}\subseteq G_i$. Since every $G_i$ is the maximal fixed subgroup of a collection of automorphisms and those satisfy a chain condition \cite[Corollary~4.2]{MV}, it follows that the groups $G_i$ eventually stabilize. Since $Y_i$ is the $G_i$-minimal subtree of $T$, it follows that the transverse coverings $\caly_i$ stabilize, as claimed. In addition $G_i$ (for sufficiently large $i$) is (up to conjugacy) the unique maximal noncyclic nonperipheral fixed subgroup of $K^0$, which concludes the proof.
\end{proof}

When $K\subseteq\ia$, the following lemma implies that $G_Y$ is also the unique maximal noncyclic, nonperipheral fixed subgroup of $K$.

\begin{lemma}\label{lemma:fixed_subgroups} 
Suppose that $K$ is a subgroup of $\ia$ and $K^0$ is finite index in $K$. Then $K$ and $K^0$ have the same fixed subgroups in $F_N$.
\end{lemma}

\begin{proof}
Any fixed subgroup of $K$ is also a fixed subgroup of $K^0$. Conversely, let $G\subseteq F_N$ be a fixed subgroup of $K^0$, and let $\phi\in K$. Then $\phi$ has a power $\phi^k\in K^0$, and $\phi^k$ preserves every conjugacy class in $G$. Since $K\subseteq\ia$, a theorem of Handel and Mosher \cite[Theorem~4.1]{HM5} ensures that $\phi$ preserves every conjugacy class in $G$. It then follows from \cite[Lemma~5.2]{MO} that $\phi_{|G}$ is a global conjugation by an element of $F_N$. This shows that $G$ is a fixed subgroup of $K$. 
\end{proof}

Our proof of Proposition~\ref{prop:stab-arat-out} relies on the following lemma. We actually work much harder: the first conclusion in this lemma (invariance by the commensurator of $K$) will only be used in the next section. We recall from the introduction that given a group $G$ and a subgroup $H\subseteq G$, the \emph{relative commensurator} $\Comm_G(H)$ is the subgroup of $G$ made of all elements $g$ such that $H\cap gHg^{-1}$ has finite index in $H$ and in $gHg^{-1}$.

\begin{lemma}\label{lemma:stab-arat-fixes-splittings} 
Let $\calf$ be a free factor system of $F_N$, let $T$ be an arational $(F_N,\calf)$-tree, and let $K\subseteq\Out(F_N,\calf)$ be a subgroup contained in the isometric stabilizer of $T$. Let $\caly$ be the $K$-canonical piecewise-$F_N$ transverse covering of $T$, and let $S$ be the skeleton of $\caly$. Then
\begin{enumerate}
\item the splitting $S$ is invariant by $\Comm_{\Out(F_N,\calf)}(K)$, and 
\item all edge stabilizers of $S$ are nontrivial and root-closed, and
\item there exists a vertex $v\in S$ whose $F_N$-orbit meets all edges of $S$, such that 
\begin{enumerate}
\item $G_v$ is finitely generated and the incident edge groups $\Inc_v$ are a nonsporadic free factor system of $G_v$, and 
\item every splitting of $F_N$ which is a blowup of $S$ at $v$ is invariant by some finite-index subgroup of $K$. 
\end{enumerate}
\end{enumerate}
\end{lemma}

\begin{proof}
 We first show that $S$ is invariant by every element $\theta\in\Comm_{\Out(F_N,\calf)}(K)$ (Property~1). Note that every edge (and therefore every vertex) stabilizer in $S$ is nontrivial, and two vertices $v_x$ and $v_Y$ are adjacent in $S$ if and only if the intersection $G_x \cap G_Y$ of their stabilizers is nontrivial (this follows from the fact that distinct free factors in $\calf$ have trivial intersection, so that an elliptic subgroup does not fix any arc of length greater than 2). Hence to show that $\theta$ preserves $S$, it is enough to show that $\theta$ preserves the conjugacy classes of the vertex stabilizers of $S$. 

 Now let $K^0$ be the finite-index subgroup of $K$ made of all automorphisms that belong to $\ia$ and preserve all orbits of directions in $T$. Note that $\Comm_{\Out(F_N,\calf)}(K)=Comm_{\Out(F_N,\calf)}(K^0)$ as $K^0$ is finite index in $K$. As the stabilizer of every vertex of the form $v_x$ is a subgroup in $\calf$, its conjugacy class is preserved by $\theta$. As $K^0$ and $\theta K^0\theta^{-1}$ are commensurable in $\ia$ they have the same fixed subgroups by Lemma~\ref{lemma:fixed_subgroups} and the conjugacy classes of these groups are permuted by $\theta$. As $G_Y$ is the unique maximal noncyclic, nonperipheral fixed subgroup of $K^0$, the automorphism $\theta$ preserves the conjugacy class of $G_Y$. Hence $\theta \cdot S = S$. 

We will now check Property~2, namely that edge stabilizers of $S$ are nontrivial and root-closed. That they are nontrivial follows from the fact that $T$ is arational (see the second conclusion of Lemma~\ref{l:claim-1}). To see that they are root-closed, it is enough to notice that stabilizers of vertices of the form $v_x$ are root-closed as they are free factors, and stabilizers of vertices of the form $v_Y$ are root-closed as they are maximal fixed subgroups (the maximal fixed subgroup of an automorphism $\alpha$ of $F_N$ is root-closed, because if $\alpha(g^k)=g^k$, then $\alpha(g)$ is the unique $k^{\text{th}}$-root of $g^k$, namely $g$). 

Now let $Y\in\caly$. By Theorem~\ref{th:gl}, the stabilizer $G_Y$ of $Y$ is finitely generated. We will now prove that, denoting by $\Inc_Y$ the collection of all $F_N$-stabilizers of edges of $S$ that are incident on $v_Y$, the Grushko deformation space of $G_Y$ relative to $\Inc_Y$ is nonsporadic (Property~$3(a)$, with $v=v_Y$). To prove this, notice that the stabilizers of vertices of the form $v_x$ form a subsystem $\calf'$ of the free factor system $\calf$, so that $\Inc_Y$ is the free factor system of $G_Y$ induced by its intersections with $\calf'$. As there exists a nonsimplicial very small $(G_Y,\Inc_Y)$-tree (namely $Y$), we deduce that $(G_Y,\Inc_Y)$ is nonsporadic. This completes our proof of Property~$3(a)$.

We will now show that given $Y\in\caly$, every blowup $\hat{S}$ of $S$ at the vertex $v_Y$ is $K^0$-invariant (Property~$3(b)$). We denote by $\hat{S}_Y$ the preimage of $v_Y$ under the collapse map $\hat{S}\to S$. Let $\alpha\in\tilde{K}^0$, let $I_\alpha$ be the induced isometry of $T$ and let $J_\alpha$ be the $\alpha$-equivariant isometry of $S$. Let $\hat{J}_\alpha:\hat{S}\to\hat{S}$ be the map defined by sending every point $x\in\hat{S}_Y$ to $g(\alpha,Y)x$, and sending every point $y$ not contained in any translate of $\hat{S}_Y$ to the unique preimage of $J_\alpha(y)$ in $\hat{S}$. We claim that $\hat{J}_\alpha$ is an $\alpha$-equivariant isometry of $\hat{S}$.

To prove that $\hat{J}_\alpha$ is an isometry, the key point is to show that if $e\subseteq S$ is an edge incident to $v$, and $x_e\in\hat{S}_Y$ is the corresponding attaching point then $\hat{J}_\alpha(x_e)=g(\alpha,Y)x_e$ is the corresponding attaching point of $J_\alpha(e)$. To check this, note that $e$ is determined by a pair $(p,Y)$, where $p \in Y$. As $I_\alpha$ acts on $Y$ by $g(\alpha, Y)$, the edge $J_\alpha(e)$ is given by the pair $(g(\alpha,Y)p,g(\alpha,Y)Y)$. Hence $J_\alpha(e)=g(\alpha,Y)e$, which has corresponding attaching point $g(\alpha,Y)x_e$ by equivariance of the blow-up.

To check that $\hat{J}_\alpha$ is $\alpha$-equivariant, it is enough to observe that for every $\alpha\in\tilde{K}$, every $Y\in\caly$, and every $h\in F_N$, one has $$\alpha(h)=g(\alpha,hY)hg(\alpha,Y)^{-1}.$$ This follows from the fact that for every $x\in Y$, one has $$g(\alpha,hY)hx=I_{\alpha}(hx)=\alpha(h)I_\alpha(x)=\alpha(h)g(\alpha,Y)x,$$ which yields the above identity as $T$ has trivial arc stabilizers. If follows that the image of $\alpha$ in $\out(F_N)$ preserves $\hat{S}$. This completes the proof of Property~$3(b)$. 
\end{proof}

\subsection{The proof of Proposition~\ref{prop:stab-arat-out}} \label{s:proof_stab_rel_aration}

\begin{lemma}\label{lemma:z-blow-up}
Let $G_v$ be a finitely generated free group, and let $\Inc_v$ be a nonempty, nonsporadic free factor system of $G_v$. Then either:
\begin{enumerate}
\item $G_v$ has a three-edge splitting relative to $\Inc_v$ with $\mathcal{Z}_{max}$ edge stabilizers and nonabelian vertex stabilizers.
\item $\Inc_v$ contains a factor $A$ isomorphic to $\mathbb{Z}$ and $G_v$ has a two-edge splitting relative to $\Inc_v$ with $\mathcal{Z}_{max}$ edge stabilizers and nonabelian vertex stabilizers.
\item $(G_v,\Inc_v)$ is isomorphic to $(F_3,\{\mathbb{Z},\mathbb{Z},\mathbb{Z}\})$, and there is a one-edge separating splitting of $G_v$ relative to $\Inc_v$ with $\mathcal{Z}_{max}$ edge stabilizers and nonabelian vertex stabilizers.
\end{enumerate}
\end{lemma}

\begin{figure}
\centering
\input{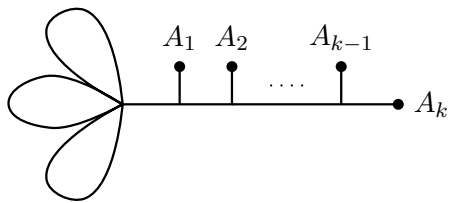}
\caption{The Grushko splitting from the proof of Lemma~\ref{lemma:z-blow-up}.}
\label{fig:grushko}
\end{figure}

\begin{proof}
 One of the following holds:
\begin{enumerate}
\item $G_v$ has a 2-edge free splitting relative to $\Inc_v$ with nonabelian vertex stabilizers, or
\item $\Inc_v$ contains a factor isomorphic to $\mathbb{Z}$ and $G_v$ has a one-edge free splitting relative to $\Inc_v$ with nonabelian vertex stabilizers, or else
\item $(G_v,\Inc_v)$ is isomorphic to $(F_3,\{\mathbb{Z},\mathbb{Z},\mathbb{Z}\})$, and $G_v$ splits as $F_2\ast\mathbb{Z}$ relative to $\Inc_v$.
\end{enumerate}
Such a splitting can be found by collapsing the Grushko splitting given in Figure~\ref{fig:grushko} (the generic situation is case 1 but there are some low-complexity examples that fall into cases 2 and 3).  One then obtains the desired $\mathcal{Z}_{max}$ splittings by folding half-edges at nonabelian vertex groups with their translate by some element of the vertex group which is not a proper power.
\end{proof}

\begin{proof}[Proof of Proposition~\ref{prop:stab-arat-out}]
 Let $S$ be the skeleton of the $K$-canonical piecewise-$F_N$ transverse covering of $T$. Let $v\in S$ be a vertex as in the third point of Lemma~\ref{lemma:stab-arat-fixes-splittings}. Property~$3(b)$ from Lemma~\ref{lemma:stab-arat-fixes-splittings} ensures that any blow-up $\hat{S}$ of $S$ at $v$ is virtually $K$-invariant. If $\hat{S}$ has nontrivial edge stabilizers, then the group of twists $\mathcal{T}$ on $\hat{S}$ is central in a finite-index subgroup of $\stab_{\ia}(\hat{S})$ (Lemma~\ref{twist-central}).  Furthermore, by \cite[Proposition~3.1]{Lev}, $\mathcal{T}$ is a free abelian group of rank $k-l$, where  $k$ is the number of $F_N$-orbits of $\mathbb{Z}$ edges in $\hat{S}$ and $l$ is the number of $F_N$-orbits of vertices with cyclic stabilizer. 
We are going to find a blow-up $\hat{S}$ at $v$ such that $\mathcal{T}$ is of rank at least 3.

We denote by $\Inc_v$ the collection of all incident edge stabilizers of $G_v$. Then $\Inc_v$ is a free factor system of $G_v$ and $G_v$ is nonsporadic relative to this free factor system. We now look at the cases given by Lemma~\ref{lemma:z-blow-up}. In the case that $G_v$ has a three-edge splitting relative to $\Inc_v$ with $\mathcal{Z}_{max}$ edge stabilizers and nonabelian vertex stabilizers, the group $\mathcal{T}$ generated by twists in these edges is $\mathbb{Z}^3$. To see this, note that as all the vertices in the splitting of $G_v$ are nonabelian, by collapsing all other edges in $\hat{S}$ we get a graph with nonabelian vertex stabilizers, three cyclic edges, and twist group $\mathcal{T}$. The same argument also shows that when $\Inc_v$ contains a cyclic factor and $G_v$ has a two-edge splitting relative to $\Inc_v$ with $\mathcal{Z}_{max}$ stabilizers and nonabelian vertex stabilizers, the group $\mathcal{T}$ generated by twists in these two edges and the twist about some incident cyclic edge is isomorphic to $\mathbb{Z}^3$. In the final case, we obtain a splitting $\hat{S}$ which collapses onto a minimal four edge $\mathcal{Z}_{max}$ splitting with five of the eight half edges based at vertices with nonabelian stabilizers. By minimality, and the fact the edge stabilizers are root-closed, at most one vertex in this splitting can be cyclic, so that the group $\mathcal{T}$ is rank at least 3.
\end{proof}

\subsection{Stabilizers in twist-rich subgroups}\label{sec:later}

Our main theorem is in the more general setting of twist-rich subgroups of $\Out(F_N)$. For this, we will also need to understand the stabilizer of an arational tree within a subgroup $\G$ of $\Out(F_N)$ which is `big enough' to satisfy the following property. 

\begin{enumerate}[($H_1$)]
 \item Given a splitting $S$ of $F_N$ with all edge stabilizers nontrivial, and a vertex $v$ of $S$ such that $G_v$ is finitely generated and the Grushko decomposition of $G_v$ relative to the incident edge groups $\Inc_v$ is nonsporadic:
\begin{enumerate}[(a)]
\item If $(G_v,\Inc_v)$ is not isomorphic to $(F_3,\{\mathbb{Z},\mathbb{Z},\mathbb{Z}\})$, then there is a blowup $S'$ of $S$ by a two-edge splitting of $(G_v,\Inc_v)$ with edge groups isomorphic to $\mathbb{Z}$ and root-closed, such that the group of twists about these edges is isomorphic to $\mathbb{Z}^2$ and $\G$ contains a finite-index subgroup of this group of twists.
\item If $(G_v,\Inc_v)$ is isomorphic to $(F_3,\{\mathbb{Z},\mathbb{Z},\mathbb{Z}\})$, then there is a blowup $S'$ of $S$ by a one-edge splitting of $(G_v,\Inc_v)$ with edge groups isomorphic to $\mathbb{Z}$ and root-closed, such that $\G$ contains a finite index subgroup of the infinite cyclic group of twists about this edge. 
\end{enumerate} 
\end{enumerate}

This will be the first property of a twist-rich subgroup. In particular, we will show in Proposition~\ref{prop:example} that ($H_1$) holds for all subgroups $\G \subseteq \Out(F_N)$  given in the main theorem of the introduction. We used cyclic blow-ups in the proof of Proposition~\ref{prop:stab-arat-out} regarding isometric stabilizers of arational trees in $\out(F_N)$, and following the same idea we will prove a slightly weaker result for isometric stabilizers of arational trees in a subgroup $\G\subseteq \Out(F_N)$ which satisfies ($H_1$). 

\begin{prop}\label{prop:stab-arat}
Let $\Gamma\subseteq\ia$ be a subgroup that satisfies Hypothesis~$(H_1)$, and let $\calf$ be a nonsporadic free factor system of $F_N$. Let $K\subseteq\Gamma\cap\Out(F_N,\calf)$ be a subgroup, and assume that some finite-index subgroup of $K$ is contained in the isometric stabilizer of an arational $(F_N,\calf)$-tree. Then 
\begin{enumerate}
\item $K$ virtually centralizes a subgroup of $\Gamma$ isomorphic to $\mathbb{Z}$.
\item One of the following two possibilities hold:
\begin{enumerate}
\item $K$ virtually centralizes a subgroup of $\Gamma$ isomorphic to $\mathbb{Z}^2$, or 
\item $\Comm_{\Gamma\cap\Out(F_N,\calf)}(K)$ contains no free abelian subgroup of rank $2N-4$. 
\end{enumerate} 
\end{enumerate}
\end{prop}

\begin{proof}
Let $T$ be an arational $(F_N,\calf)$-tree such that $K$ is contained in the isometric stabilizer of $T$. Let $S$ be the skeleton of the $K$-canonical piecewise-$F_N$ transverse covering of $T$, and let $v\in S$ be a vertex of $S$ given by Lemma~\ref{lemma:stab-arat-fixes-splittings}. We denote by $\Inc_v$ the collection of all incident edge stabilizers.

Hypothesis~$(H_1)$ ensures that we can find a blowup $\hat{S}$ of $S$ at $v$ by a cyclic edge such that the group of twists $\mathcal{T}$ associated to this edge intersects $\Gamma$ nontrivially. Property~$3(b)$ from Lemma~\ref{lemma:stab-arat-fixes-splittings} ensures that $\hat{S}$ is virtually $K$-invariant. Therefore $K$ virtually centralizes $\Gamma\cap\mathcal{T}$, which is isomorphic to $\mathbb{Z}$. This proves the first assertion of the lemma.

If $(G_v,\Inc_v)$ is not isomorphic to $(F_3,\{\mathbb{Z},\mathbb{Z},\mathbb{Z}\})$, then Hypothesis~$(H_1)$ ensures that we can find a blowup $\hat{S}$ of $S$ at $v$ by two cyclic edges, such that the group of twists associated to these two edges is isomorphic to $\mathbb{Z}^2$ and has a finite-index subgroup contained in $\Gamma$. The same argument as above ensures that in this case, $K$ virtually centralizes a subgroup of $\Gamma$ isomorphic to $\mathbb{Z}^2$. 

We can therefore assume that $(G_v,\Inc_v)$ is isomorphic to $(F_3,\{\mathbb{Z},\mathbb{Z},\mathbb{Z}\})$. Let $\overline{S}$ be the tree obtained from $S$ by collapsing all edges of $S$ whose stabilizer is not isomorphic to $\mathbb{Z}$. As $S$ is $\Comm_{\Out(F_N,\calf)}(K)$-invariant (by the first point in Lemma~\ref{lemma:stab-arat-fixes-splittings}), so is $\overline{S}$. Using \cite[Proposition~3.1]{Lev}, we see that the group of twists $\mathcal{T}_{\overline{S}}$ of $\overline{S}$ (in $\Out(F_N)$) contains a subgroup isomorphic to $\mathbb{Z}^2$ or $\mathbb{Z}^3$ (given by the incident edges at $v$). If $\Gamma\cap\mathcal{T}_{\overline{S}}$ contains a subgroup isomorphic to $\mathbb{Z}$, then by blowing-up a cyclic edge at $v$ as above, we see that $K$ virtually centralizes a subgroup of $\Gamma$ isomorphic to $\mathbb{Z}^2$. Otherwise, a maximal free abelian subgroup of $\Stab_{\Gamma}(\overline{S})$ has rank at most $(2N-3)-2=2N-5$. As $\Comm_{\Gamma\cap\Out(F_N,\calf)}(K)\subseteq\Stab_{\Gamma}(\overline{S})$, we deduce in particular that $\Comm_{\Gamma\cap\Out(F_N,\calf)}(K)$ has no free abelian subgroup of rank $2N-4$. 
\end{proof}

\section{Direct products of free groups in $\Out(F_N)$}\label{sec:product-f2}

\emph{In this section we will use direct products of free groups in $\Out(F_N)$ to distinguish between stabilizers of separating and nonseparating one-edge free splittings. They will also be used to see if two one-edge nonseparating free splittings span an edge in $\ens$.} 
\\
\\
\indent Given a group $G$, we denote by $\pr(G)$ the maximal integer $k$ such that $G$ contains a subgroup isomorphic to a direct product of $k$ nonabelian free groups. Note that passing to a finite index subgroup does not change $\pr(G)$. In this section we shall show that $\pr(\Out(F_N))=2N-4$ for all $N\ge 3$ and study $\pr(G)$ when $G$ is the stabilizer of a free splitting. 

 A typical example of a direct product of $2N-4$ nonabelian free groups in $\out(F_N)$ is given as follows. Pick a basis $x_1, x_2, \ldots, x_N$ of $F_N$. For every $i\in\{3,\dots,N\}$, the subgroup $L_i$ made of all automorphisms of the form $x_i\mapsto l_ix_i$ with $l_i$ varying in $\langle x_1,x_2\rangle$ is free. Likewise, for every $i\in\{3,\dots,N\}$, the subgroup $R_i$ made of all automorphisms of the form $x_i\mapsto x_ir_i$ with $r_i$ varying in $\langle x_1,x_2\rangle$ is free. The groups $L_i$ and $R_i$ pairwise commute, giving a direct product of $2N-4$ nonabelian free groups in $\Out(F_N)$. This direct product of free groups is equal to the group of twists in the stabilizer of the free splitting given by the rose with $N-2$ petals corresponding to $x_3, \ldots, x_N$ and vertex group $\langle x_1, x_2 \rangle$. 

Every inner automorphism given by an element of $\langle x_1, x_2 \rangle$ commutes with the examples above, which yields a direct product of  $2N-3$ copies of $F_2$ in $\aut(F_N)$. A complete classification of these maximal direct products is beyond the scope of this paper, however we will need to show that these examples are maximal.

\begin{theo}\label{direct_products_of_free_groups}
For every $N\ge 2$, we have $\pr(\Aut(F_N))=2N-3$. 
\\ For every $N\ge 3$, we have $\pr(\Out(F_N))=2N-4$. 
\\ In addition, if $H$ is a subgroup of $\Out(F_N)$ isomorphic to a direct product of $2N-4$ nonabelian free groups, then $H$  virtually fixes a one-edge nonseparating free splitting of $F_N$, but  does not virtually fix any one-edge separating   free splitting of $F_N$.
\end{theo}

We will prove Theorem~\ref{direct_products_of_free_groups} by induction on the rank. The base case where $N=2$ is given by the following lemma.

\begin{lemma}\label{lemma:aut-f2}
The group $\aut(F_2)$ does not contain a direct product of two nonabelian free groups.
\end{lemma}

\begin{proof}
Suppose that $H=H_1 \times H_2$ is a direct product of two nonabelian free groups in $\aut(F_2)$. As both the kernel and quotient are virtually free in the exact sequence $1 \to F_2 \to \Aut(F_2) \to \Out(F_2) \to 1$, the image of some factor ($H_1$, say) is finite in $\out(F_2)$ and $H_1$ intersects $F_2$ in a nonabelian subgroup. It follows that the other factor $H_2$ embeds in $\out(F_2)$ under the quotient map. If $\phi$ is an automorphism in $H_2$ then $\phi$ commutes with every $\ad_x \in H_1$. This implies that $\phi(x)=x$ for every $x\in H_1$. In particular, $\phi$ has a nonabelian fixed subgroup.  
By using the identification of $\out(F_2)$ with the mapping class group of a once-holed torus, we see that $H_2$ cannot contain any exponentially growing elements and either $H_2$ is finite or virtually cyclic and generated by a power of a Dehn twist, which is a contradiction.
\end{proof}

Our proof of Theorem~\ref{direct_products_of_free_groups} relies on three more lemmas.

\begin{lemma}\label{direct-product-exact-seq}
Let $1 \to K \to G \to Q \to 1$ be an exact sequence of groups. Then $\pr(G)\le\pr(K)+\pr(Q)$.
\end{lemma}

\begin{proof}
Let $H=H_1 \times H_2 \times \cdots \times H_k$ be a direct product of nonabelian free groups in $G$, and let $H_K=H \cap K$ be the normal subgroup of $H$ contained in the kernel. If $x=(h_1, h_2,\ldots,h_k)$ belongs to $H_K$, then by normality, so does $y=(gh_1g^{-1},h_2,\ldots,h_k)$ for every $g \in H_1$. Then $yx^{-1}=(gh_1g^{-1}h_1^{-1},1,1,\ldots,1)$ is also in $H_K$. This calculation implies that if the projection of $H_K$ to some factor is nontrivial then $H_K$ intersects that factor in a nonabelian free group. Hence there can be at most $\pr(K)$ factors with nontrivial projections of $H_K$, and the direct product of the remaining $k - \pr(K)$ factors embed in $Q$. This implies that $k - \pr(K) \leq \pr(Q)$  and the result follows.
\end{proof}

\begin{lemma}\label{automorphic-lifts} 
If $S$ is a one-edge nonseparating free splitting of $F_N$ then $\stab(S)$ has an index 2 subgroup $\stab^0(S)$ with a split-exact sequence \[ 1\to F_{N-1}  \to \stab^0(S) \to \Aut(F_{N-1}) \to 1.\] If $S$ is a one-edge separating free splitting of $F_N$ corresponding to $A\ast B$ then $\stab(S)$ has a  subgroup $\stab^0(S)$ of index at most 2 such that \[ \stab^0(S) \cong \Aut(A) \times \Aut(B). \] If $S$ is a free splitting of $F_N$ such that $S/F_N$ is a two-edge loop with vertex groups $A$ and $B$, then $\stab(S)$ has a  subgroup $\stab^0(S)$ of index at most 4 with a split-exact sequence \[ 1\to A \times B \to \stab^0(S) \to \Aut(A) \times \Aut(B)\to 1. \]
\end{lemma}

\begin{proof}
This will be familiar to some readers. The proofs of the first two parts can be found in Section 1.4 of \cite{GS}, for example. In short, stabilizers of free splittings in $\Out(F_N)$ have very nice \emph{automorphic lifts} to subgroups of $\aut(F_N)$. We give a proof of the third statement along these lines. Let $S$ be a two-edge loop splitting of $F_N$. Let $e$ be an edge of $S$ with endpoints $v_A$ and $v_B$ with stabilizers $A$ and $B$, respectively. The subgroup $\stab^0(S)$ which acts trivially on the quotient graph $S/F_N$ is of index at most 4. The preimage $\tilde K$ of $\stab^0(S)$ in $\Aut(F_N)$ acts on the tree $S$. If $\tilde K_e$ is the stabilizer of $e$ in $\tilde K$, then the map $\tilde K_e \to \stab^0(S)$ induced by the map $\Aut(F_N) \to \Out(F_N)$ is an isomorphism (it is injective as no nontrivial inner automorphism fixes $e$ and is surjective as every element of $\stab^0(S)$ has a representative in $\Aut(F_N)$ fixing $e$ as the action on $S/F_N$ is trivial). There is a natural map from $\tilde K_e$ to $\Aut(A) \times \Aut(B)$ given by restriction of an automorphism to its action on the vertex groups. We claim that the kernel of this map is isomorphic to $A \times B$. Indeed, if $e'$ is an edge in a distinct orbit to $e$ at $v_A$ (i.e. representing the other edge in the loop) and $t$ is an element taking $e'$ to an edge $te'$ adjacent to $v_B$ then $F_N \cong A\ast B \ast \langle t \rangle $. Suppose $\alpha \in \tilde K_e$, and let $I_\alpha$ be the induced action on the tree. Then $I_\alpha(e')=ae'$ for some $a \in A$ and $I_\alpha(te')=bte'$ for some $b \in B$ and $I_\alpha(te')=\alpha(t)I_\alpha(e')=\alpha(t)ae'$, which implies that $\alpha(t)a=bt$ and $\alpha(t)=bta^{-1}$. This gives a way of identifying the kernel of the map to $\Aut(A)\times \Aut(B)$ with $A \times B$. The decomposition $F_N=A\ast B\ast\langle t\rangle$ gives a map from $\Aut(A)\times\Aut(B)$ to $\Aut(F_N)$, showing that the exact sequence is split.
\end{proof}

\begin{lemma}\label{direct_product_rel_arational}
Let $H$ be a direct product of $2N-5$ nonabelian free groups contained in $\out(F_N)$. Then no finite index subgroup of $H$ is contained in the homothetic stabilizer of a relatively arational tree. 
\end{lemma}

\begin{proof}
Assume towards a contradiction that $H$ contains a finite-index subgroup $H'$ contained in the homothetic stabilizer of a relatively arational tree $T$. Then $H'$ has a morphism onto $\mathbb{R}_+^\ast$ whose kernel $K$ is contained in the isometric stabilizer of $T$, and $K$ also contains a direct product of $2N-5$ nonabelian free groups. Proposition~\ref{prop:stab-arat-out} implies that $K$ centralizes a subgroup of $\Out(F_N)$ isomorphic to $\mathbb{Z}^3$, and this implies that $\Out(F_N)$ contains a free abelian subgroup of rank $2N-2$, a contradiction.  
\end{proof}

\begin{proof}[Proof of Theorem~\ref{direct_products_of_free_groups}]
We argue by induction on $N$. The base case $N=2$ was treated in Lemma~\ref{lemma:aut-f2}. 

Let $H=H_1 \times H_2 \times \cdots \times H_k$ be a subgroup of $\ia$ which is a direct product of $k$ nonabelian free groups, with $k \geq 2N-4$. Let $\mathcal{F}$ be a maximal $H$-invariant free factor system. 

We first assume that $\mathcal{F}$ is nonsporadic, and aim for a contradiction in this case. Using Proposition~\ref{prop:maximal-unbounded}, maximality of $\mathcal{F}$ implies that the group $H$ acts on $\FF=\FF(F_N,\calf)$ with unbounded orbits. Proposition~\ref{product-vs-hyp} then implies that after possibly reordering the factors the subgroup $H'=H_1 \times H_2 \times \cdots \times H_{k-1}$ has a finite orbit in $\partial_\infty\FF$. By Proposition~\ref{prop:fix-boundary}, this implies that $H'$ has a finite-index subgroup that fixes the homothety class of a relatively arational tree, contradicting Lemma~\ref{direct_product_rel_arational}. 

Therefore $\calf$ is sporadic, which implies that $H$ fixes a free splitting of $F_N$. We first assume that $H$ fixes a separating free splitting of $F_N$, which is the Bass--Serre tree of a free product decomposition $F_N=A\ast B$, and aim for a contradiction. Then by the second part of Lemma~\ref{automorphic-lifts} the group $H$ has a finite-index subgroup that embeds into $\Aut(A)\times\Aut(B)$. If both $A$ and $B$ are noncyclic, then by induction we have $\pr(H)\le (2\mathrm{rk}(A)-3)+(2\mathrm{rk}(B)-3)=2N-6$. If $A$ is cyclic, then by induction $\pr(H)\le 2(N-1)-3=2N-5$. In both cases, we have reached a contradiction.

Therefore $H$ fixes a nonseparating free splitting of $F_N$, which is the Bass--Serre tree of a HNN extension $F_N=C\ast$.  By the first part of Lemma~\ref{automorphic-lifts}, the group $H$ has a finite-index subgroup that maps to $\Aut(C)$, with kernel contained in $C$. Using Lemma~\ref{direct-product-exact-seq} and arguing by induction, we deduce that $\pr(H)\le 2(N-1)-3+1=2N-4$.

We have thus proved that $\pr(\Out(F_N))=2N-4$. The result for $\Aut(F_N)$ follows, using the short exact sequence $1\to F_N\to\Aut(F_N)\to\Out(F_N)\to 1$ and Lemma~\ref{direct-product-exact-seq}. 
\end{proof}

We will also need to look at direct products of free groups in $\out(F_N)$ that fix a two-edge loop splitting of $F_N$.

\begin{lemma}\label{lemma:product-in-circle}
Let $N\ge 3$, and let $S$ be a free splitting of $F_N$ such that $S/F_N$ is a two-edge loop.
\\ Then $\pr(\Stab(S))\le 2N-6$.
\end{lemma}

\begin{proof}
Let $A$ and $B$ be the vertex groups of $S/F_N$. Then by Lemma~\ref{automorphic-lifts} the group $\Stab(S)$ has a finite index subgroup $\stab^0(S)$ fitting in the exact sequence 
\[ 1\to A \times B \to \stab^0(S) \to \Aut(A) \times \Aut(B)\to 1. \]
Let $k:=\mathrm{rk}(A)$ (so that $\mathrm{rk}(B)=N-k-1$). If both $A$ and $B$ have rank at least $2$, using Theorem~\ref{direct_products_of_free_groups} and Lemma~\ref{direct-product-exact-seq}, we deduce that $\pr(\Stab(S))\le (2k-3)+(2(N-k-1)-3)+2=2N-6$. If $A$ is cyclic and $B$ is noncyclic, we deduce that $\pr(\stab(S))\le 2(N-2)-3+1=2N-6$. If both $A$ and $B$ are cyclic (in rank $N=3$), then $\Stab(S)$ is virtually abelian and the result also holds in this case.
\end{proof}

\section{Twist-rich subgroups of $\ia$}\label{sec:hypotheses}

\emph{In this section, we introduce the notion of \emph{twist-rich} subgroups of $\out(F_N)$, which will be the subgroups to which our methods apply. In particular, we will show that all the subgroups of $\out(F_N)$ mentioned in the introduction are twist-rich. As mentioned previously, to avoid periodic behaviour we work in the finite-index subgroup $\ia$ of $\out(F_N)$.}

\subsection{Definition}

\begin{de}[\textbf{\emph{Twist-rich subgroups of $\ia$}}]
A subgroup $\Gamma$ of $\ia$ is \emph{twist-rich} if it satisfies the following conditions:
\begin{enumerate}[($H_1$)] 
 \item  Given a splitting $S$ of $F_N$ with all edge stabilizers nontrivial, and a vertex $v$ of $S$ such that $G_v$ is finitely generated and the Grushko decomposition of $G_v$ relative to the incident edge groups $\Inc_v$ is nonsporadic:
\begin{enumerate}[(a)]
\item If $(G_v,\Inc_v)$ is not isomorphic to $(F_3,\{\mathbb{Z},\mathbb{Z},\mathbb{Z}\})$, then there is a blowup $S'$ of $S$ by a two-edge splitting of $(G_v,\Inc_v)$ with edge groups isomorphic to $\mathbb{Z}$ and root-closed, such that the group of twists about these edges is isomorphic to $\mathbb{Z}^2$ and $\G$ contains a finite-index subgroup of this group of twists.
\item If $(G_v,\Inc_v)$ is isomorphic to $(F_3,\{\mathbb{Z},\mathbb{Z},\mathbb{Z}\})$, then there is a blowup $S'$ of $S$ by a one-edge splitting of $(G_v,\Inc_v)$ with edge groups isomorphic to $\mathbb{Z}$ and root-closed, such that $\G$ contains a finite index subgroup of the infinite cyclic group of twists about this edge. 
\end{enumerate}
\item For every free splitting $S$ and every half-edge $e$ incident on a vertex $v$ with nonabelian stabilizer $G_v$, the intersection of $\Gamma$ with the group of twists about $e$ is nonabelian and viewed as a subgroup of $G_v$, it is not elliptic in any $\zmax$ splitting of $G_v$. 
\end{enumerate}
\end{de}

Let us provide some intuition for this definition. Hypothesis~$(H_1)$ has already appeared in Section~\ref{sec:later} and is used in the study of $\Gamma$-stabilizers of relatively arational trees. Hypothesis~$(H_2)$ -- which we believe is the most crucial of the two -- is here to ensure that $\Gamma$ intersects the stabilizer of a free splitting $S$ in a large enough subgroup. Importantly for us, $(H_2)$ implies that stabilizers of one-edge nonseparating splittings in $\G$ contain direct products of nonabelian free groups coming from twists. We take advantage of these direct products of free groups to give an algebraic characterization of $\Gamma$-stabilizers of one-edge nonseparating free splittings. Furthermore, we will see in Section~\ref{sec:twist-rich-unique-splitting} that the large group of twists can be combined with the methods of Cohen--Lustig from Lemma~\ref{twist-compatible} to show that the $\Gamma$-stabilizer of a one-edge nonseparating splitting does not fix any other free splitting.

Notice that if $\Gamma\subseteq\Gamma'$ are subgroups of $\ia$, and $\Gamma$ is twist-rich, then $\Gamma'$ is twist-rich. We shall see later that if $\Gamma'$ is twist-rich and $\Gamma$ is a  finite-index subgroup of $\Gamma'$, then $\Gamma$ is also twist rich. 

\subsection{Properties of $\zmax$ splittings and $\zmax$-factors} 

A \emph{$\zmax$-factor} of $F_N$ is a vertex stabilizer of a $\zmax$ splitting. It is \emph{proper} if it is nontrivial and not equal to $F_N$. Such subgroups appear naturally in the context of fixed elements of automorphisms, for instance: 

\begin{proposition}[{\cite[Theorem~7.14]{GL-aut}}] \label{p:zmax_rigid}
Let $g \in F_N$. Then the subgroup $\Out(F_N;\langle g \rangle)$ of automorphisms which preserve $\langle g \rangle$ up to conjugacy is infinite if and only if $g$ is contained in a proper $\zmax$-factor of $F_N$.
\end{proposition}

We outline some basic facts about $\zmax$-factors below. 

\begin{proposition}\label{prop:zmax_factors} $\zmax$-factors satisfy the following properties. 
\begin{enumerate}
 \item There exists $g \in F_N$ which is not contained in any proper $\zmax$-factor of $F_N$.
 \item $\zmax$-factors of $F_N$ satisfy the bounded ascending chain condition. Explicitly, every strictly ascending chain $G_1 \subsetneq G_2 \subsetneq \cdots \subsetneq G_k $ of $\zmax$-factors of $F_N$ has size $k\le 2N$.
 \item If a subgroup $K \subseteq F_N$ is not contained in any proper $\zmax$-factor of $F_N$ and $P$ is either finite index in $K$ or a nontrivial normal subgroup of $K$, then $P$ is not contained in any proper $\zmax$-factor of $F_N$. 
 \item A subgroup $K \subseteq F_N$ is contained in a proper $\zmax$-factor of $F_N$ if and only if every element of $K$ is contained in a proper $\zmax$-factor.
\end{enumerate}
\end{proposition}

\begin{proof}
By Proposition~\ref{p:zmax_rigid}, if only finitely many automorphisms preserve the conjugacy class of an element, then this element is not contained in a proper $\zmax$-factor. The existence of such an element is a consequence of Whitehead's algorithm (\cite{Whi}, see also \cite{Sta}). For instance, one can take $g=x_1^3x_2^3\cdots x_N^3$ if $x_1, x_2, \ldots x_N$ is a basis of $F_N$. 

For the ascending chain condition, every $\zmax$-factor is a maximal fixed subgroup of an automorphism (e.g. one obtained by twisting about all adjacent edges in a splitting where this factor is a vertex \cite{CL2}). By \cite{MV}, any strictly ascending chain of fixed subgroups has length at most $2N$.

For Part 3, the conclusion is clear if $K$ is cyclic, so we can assume it is not. As every finite index subgroup of $K$ contains a nontrivial normal subgroup  of $K$ we may focus on the case where $P$ is a nontrivial normal subgroup of $K$. Then $P$ is noncyclic. If $P$ is contained in a $\zmax$-factor of $F_N$, then there exists a $\zmax$ splitting $S$ of $F_N$ such that $P$ is elliptic in $S$. As $S$ has cyclic edge stabilizers, $P$ fixes a unique vertex $x$ in $S$. As $P$ is normal in $K$, if $h \in K$ then $hx$ is also fixed by $P$, so $hx=x$. Therefore $x$ is fixed by $K$, which is a contradiction as $K$ is not contained in a $\zmax$-factor of $F_N$.

For Part 4, it is clear that if $K$ is contained in a proper $\zmax$-factor then so is every element of $K$. To prove the converse we assume that $K$ is not contained in a proper $\zmax$-factor and claim that there exists $g \in K$ that is not contained in a proper $\zmax$-factor. As there is a bound on the length of an increasing chain of $\zmax$-factors of $F_N$, the group $K$ contains a finitely generated subgroup $K'$ which is not contained in any proper $\zmax$-factor of $F_N$. By Part 1, there exists $g\in K'$ such that $g$ is not contained in a proper $\zmax$-factor of $K'$. Let $S$ be a $\zmax$ splitting of $F_N$. As $K'$ is not contained in any $\zmax$-factor of $F_N$, the group $K'$ has a well-defined, nontrivial minimal subtree $S_{K'}$ with respect to its action on $S$. As $S$ is a $\zmax$ splitting of $F_N$, it follows that $S_{K'}$ is a $\zmax$ splitting of $K'$. As $g$ is not contained in any $\zmax$-factor of $K'$, it follows that $g$ is a hyperbolic isometry of $S_{K'}$ and is not elliptic in $S$. As $S$ was chosen arbitrarily, it follows that $g$ is not contained in any $\zmax$-factor of $F_N$.
\end{proof}

Part 3 of the above proposition implies that if $P$ is obtained from $K$ by passing to a finite-index or a proper normal subgroup a finite number of times, then $P$ is elliptic in some $\zmax$ splitting of $F_N$ if and only if $K$ is.

\begin{proposition} \label{prop:fi-twist-rich}
Suppose that $\G$ is twist-rich and $\G'$ is a finite-index subgroup of $\G$. Then $\G'$ is twist-rich.
\end{proposition}

\begin{proof}
The fact that $\G'$ satisfies $(H_1)$ is immediate from the definition, and $(H_2)$ follows by Part 3 of Proposition~\ref{prop:zmax_factors}.
\end{proof}

\subsection{Stabilizers of free splittings in twist-rich subgroups}\label{sec:twist-rich-unique-splitting}

The purpose of this section is to show that the stabilizer of a free splitting $S$ in a twist-rich subgroup only fixes the obvious free splittings of $F_N$ given by collapses of $S$. 

\begin{lemma}\label{lemma:single-splitting-stabilized}
Let $\Gamma\subseteq\ia$ be a twist-rich subgroup. Let $S$ be a free splitting of $F_N$ such that every vertex of $S$ has nonabelian stabilizer, let $K:=\Stab_{\Gamma}(S)$, and let $K'$ be a finite-index subgroup of $K$.
\\ Then every $K'$-invariant free splitting of $F_N$ is a collapse of $S$.
\end{lemma}

\begin{proof}
Since $K\subseteq\ia$, every $K'$-invariant free splitting is $K$-invariant, so we can assume without loss of generality that $K'=K$. For every half-edge $e$ of $S$ incident on a vertex $v$, choose an element $z_e\in G_v$ which is not a proper power, such that $G_v$ is freely indecomposable relative to $z_e$, and such that the corresponding twist is contained in $\Gamma$ (this exists in view of Hypothesis~$(H_2)$ from the definition of a twist-rich subgroup together with the fourth part of Proposition~\ref{prop:zmax_factors}). Let $S'$ be the splitting obtained from $S$ by folding every half-edge $e$ with its translate by $z_e$. Notice that $S'$ can be viewed as a bipartite tree on the vertex set $V_0\cup V_1$, where $V_0$ corresponds to vertices of $S$, and $V_1$ corresponds to midpoints of edges of $S$. For every $v\in V_0$, the group $G_v$ is freely indecomposable relative to the incident edge stabilizers. For every $v\in V_1$, the group $G_v$ is isomorphic to $F_2$, generated by the two incident edge groups. If $U$ is a $K$-invariant free splitting, then Lemma~\ref{twist-compatible} implies that $U$ is compatible with every one edge collapse of $S'$, and therefore $S'$ itself (see \cite[Proposition~A.17]{GL-jsj}). But in view of the above description of $S'$, every free splitting compatible with $S'$ is a collapse of $S$.   
\end{proof}

For future use, we mention that the same argument also yields the following two variations over the previous statement. 

\begin{lemma}\label{lemma:single-splitting-stabilized-2}
Let $\Gamma\subseteq\ia$ be a subgroup that contains a power of every Dehn twist. Let $S$ be a free splitting of $F_N$ such that every vertex of $S$ has nontrivial stabilizer, let $K:=\Stab_{\Gamma}(S)$, and let $K'$ be a finite-index subgroup of $K$.
\\ Then every $K'$-invariant free splitting of $F_N$ is a collapse of $S$.
\end{lemma}

\begin{proof}
In the above proof, the fact that vertex stabilizers were nonabelian as opposed to just nontrivial was only used to ensure that the corresponding twists are contained in $\Gamma$, which is automatic (up to passing to a power) here. The proof of Lemma~\ref{lemma:single-splitting-stabilized} thus carries over to yield Lemma~\ref{lemma:single-splitting-stabilized-2}.
\end{proof}

\begin{lemma}\label{lemma:single-splitting-2}
Let $S$ be a one-edge nonseparating free splitting of $F_N$, and let $K\subseteq\Stab_{\Out(F_N)}(S)$ be a group that contains a twist about a half-edge of $S$ whose twistor is not contained in any proper free factor of the incident vertex group.
\\ Then $S$ is the only nontrivial $K$-invariant free splitting of $F_N$.
\qed 
\end{lemma}

\subsection{Examples of twist-rich subgroups}

\begin{proposition}\label{prop:example}
Let $N\ge 3$. Then every subgroup of $\ia$ which contains a term of the Andreadakis--Johnson filtration of $\Out(F_N)$ is twist-rich.
\end{proposition}

\begin{proof}
Let $k\ge 1$, and assume that $\Gamma$ contains the $k^{\text{th}}$ term of the Andreadakis--Johnson filtration of $\Out(F_N)$. 

We first prove Hypothesis~$(H_1)$. Let $S$ be a splitting of $F_N$, and let $v\in S$ be a vertex such that $(G_v,\Inc_v)$ is nonsporadic. We denote by $\calf$ the smallest free factor system of $G_v$ that contains $\Inc_v$.

If $(G_v,\Inc_v)$ is not of the form $(F_3,\{\mathbb{Z},\mathbb{Z},\mathbb{Z}\})$, then $(G_v,\calf)$ is not of the form $(F_3,\{\mathbb{Z},\mathbb{Z},\mathbb{Z}\})$ either. Therefore, there exists a nontrivial free splitting $S_v$ of $G_v$ relative to $\Inc_v$ in which every vertex stabilizer is nonabelian (cf.\ the proof of Lemma~\ref{lemma:z-blow-up}).  For every half-edge $e$ of $S_v$, denoting by $w$ the vertex of $S_v$ incident on $e$, we can choose an element $g_e$ in the $k^{\text{th}}$ derived subgroup of $G_w$. Then $g_e$ is also in the $k^{\text{th}}$ derived subgroup of $F_N$. This implies that the twist by $g_e$ around $e$, viewed as an automorphism of $F_N$ after blowing up $S$ at $v$ into $S_v$, belongs to $\Gamma$ (it is either a partial conjugation by $g_e$ or a transvection of some basis element by $g_e$). By considering two half-edges $e$ and $e'$ in distinct orbits, we thus get a free abelian group of twists isomorphic to $\mathbb{Z}^2$ contained in $\Gamma$.

If $(G_v,\Inc_v)$ is of the form $(F_3,\{\mathbb{Z},\mathbb{Z},\mathbb{Z}\})$, then we can only assume that one of the vertex groups of $S_v/G_v$ is nonabelian, and consider a twist as above around a half-edge incident on $e$.

To prove $(H_2)$, notice that the group of twists about $e$ in $\Gamma$ contains the $k^{\text{th}}$ derived subgroup of $G_v$. As this is a normal subgroup of $G_v$, the fact that it is not elliptic in any nontrivial $\zmax$ splitting of $G_v$ follows from Part 3 of Proposition~\ref{prop:zmax_factors}.     
\end{proof}

We also record the following class of examples, for which twist-richness is clear from the definition.

\begin{prop}\label{prop:example2}
Let $N\ge 3$, and let $\Gamma$ be a subgroup of $\ia$ such that every twist has a power contained in $\Gamma$.
\\ Then $\Gamma$ is twist-rich.
\qed
\end{prop}

\begin{remark}
As mentioned in the introduction, this applies for example to the kernel of the natural morphism from $\Out(F_N)$ to the outer automorphism group of a free Burnside group of rank $N$ and any exponent.
\end{remark}

\section{Characterizing stabilizers of nonseparating free splittings}\label{sec:vertices}

\emph{Let $\Gamma$ be a twist-rich subgroup of $\ia$. The goal of the present section is to prove that the set of commensurability classes of $\Gamma$-stabilizers of one-edge nonseparating free splittings of $F_N$ is $\Comm(\Gamma)$-invariant. In other words $\Comm(\Gamma)$ preserves the set of commensurability classes of stabilizers of vertices of $\ens$.} 
\\
\\
\indent We introduce the following algebraic property of a subgroup $H\subseteq\Gamma$.

\begin{itemize}
\item[$(P_{\stab})$] The group $H$ satisfies the following two properties:
\begin{enumerate}
\item $H$ contains a normal subgroup that splits as a direct product $K_1\times K_2$ of two nonabelian free groups, such that for every $i\in\{1,2\}$, if $P_i$ is a normal subgroup of a finite index subgroup of $K_i$, then $C_\Gamma(P_i)=K_{i+1}$ (where indices are taken mod $2$).  
\item $H$ contains a direct product of $2N-4$ nonabelian free groups.
\end{enumerate} 
\end{itemize}

In Section~\ref{sec:prop-satisfied}, we will check that the $\Gamma$-stabilizer of a one-edge nonseparating free splitting $S$ satisfies Property~$\Pstab$ (by taking for $K_1$ and $K_2$ the intersections of $\Gamma$ with the groups of left and right twists about the splitting $S$). In Section~\ref{sec:converse}, we will show that conversely, every subgroup of $\Gamma$ which satisfies Property~$\Pstab$ fixes a one-edge nonseparating free splitting. This will be enough to prove in Section~\ref{sec:ccl} that $\Comm(\Gamma)$ preserves the set of commensurability classes of stabilizers of one-edge nonseparating free splittings.

\subsection{Stabilizers of nonseparating free splittings satisfy $\Pstab$.} \label{sec:prop-satisfied}

We will now prove the following proposition.

\begin{prop}\label{prop:property-satisfied}
Let $\Gamma$ be a subgroup of $\ia$ which satisfies Hypothesis~$(H_2)$, and let $S$ be a one-edge nonseparating free splitting of $F_N$.
\\ Then $\Stab_\Gamma(S)$ satisfies Property~$(P_{\stab})$.
\end{prop}

In order to prove Proposition~\ref{prop:property-satisfied}, we need to understand centralizers of half-groups of twists in $\Gamma$. Let $S$ be a one-edge nonseparating free splitting of $F_N$, and let $A\subseteq F_N$ be a corank one free factor such that $S$ is the Bass--Serre tree of the HNN extension $F_N=A\ast$. Let $\Stab^0(S)$ be the index $2$ subgroup of $\Stab_{\Out(F_N)}(S)$ made of automorphisms acting trivially on the quotient graph $S/F_N$, i.e.\ those that do not flip the unique edge in this graph. We mention that $\stab_{\ia}(S)\subseteq\stab^0(S)$ (Lemma~\ref{lemma:stab-splitting-ia}). Then $\Stab^0(S)$ surjects onto $\Out(A)$, and the kernel of this map is precisely equal to the group of twists of the splitting $S$. Let $e_1$ and $e_2$ be the two half-edges of $S/F_N$, and for every $i\in\{1,2\}$, let $K_{e_i}$ be the group of twists (in $\Out(F_N)$) about the edge $e_i$, which is isomorphic to $A$. 
We will call $K_{e_1}$ the \emph{group of left twists} of $S$, and $K_{e_2}$ the \emph{group of right twists} of $S$. 
By \cite[Proposition~3.1]{Lev}, the group of twists of the splitting $S$ is isomorphic to $K_{e_1}\times K_{e_2}$.  This gives a short exact sequence \[1 \to K_{e_1} \times K_{e_2} \to \Stab^0(S) \to \Out(A) \to 1\] describing the automorphisms fixing $S$ and acting trivially on $S/F_N$. 

\begin{proof}[Proof of Proposition~\ref{prop:property-satisfied}]
The fact that $\Stab_{\Gamma}(S)$ contains a direct product of $2N-4$ nonabelian free groups follows from Hypothesis~$(H_2)$: indeed, one can find a blowup $\hat{S}$ of $S$ which is a rose with $N-2$ petals, and Hypothesis~$(H_2)$ ensures that $\Stab_{\Gamma}(\hat{S})$ contains a direct product of $2N-4$ nonabelian free groups. As subgroups of $\ia$ preserve $F_N$-orbits of edges, $\Stab_{\Gamma}(\hat{S})$ is contained in $\Stab_{\Gamma}(S)$  (Lemma~\ref{lemma:stab-splitting-ia}). 

We will now prove that $\Stab_{\Gamma}(S)$ satisfies the first assertion from Property~$(P_\stab)$. As $K_{e_1}$ and $K_{e_2}$ are normal subgroups of $\Stab^0(S)$, the groups $K_1=K_{e_1} \cap \Gamma$ and $K_2=K_{e_2} \cap \Gamma$ are normal subgroups of $\Stab_\G(S)$ ($K_1$ and $K_2$ are the intersections of $\Gamma$ with the groups of left and right twists about $S$, respectively). Then $K_1\times K_2$ is a normal subgroup of $\Stab_{\Gamma}(S)$. Let $K'_1$ be a finite-index subgroup of $K_1$, and let $P_1$ be a normal subgroup of $K'_1$. We aim to prove that $C_{\Gamma}(P_1)=K_2$ (by symmetry, the same will hold true if we reverse the roles of $K_1$ and $K_2$).

It is clear that every right twist about $S$ centralizes $P_1$. We need to prove that conversely $C_{\Gamma}(P_1)$ is contained in the group of right twists of the splitting $S$. 
Let $A\subseteq F_N$ be a corank one free factor such that $S$ is the Bass--Serre tree of the splitting $F_N=A\ast$. We identify the group of left twists about $S$ (in $\Out(F_N)$) with $A$. Hypothesis~$(H_2)$ shows that $K_1$ is not contained in any proper $\zmax$-factor of $A$. Part 3 of Proposition~\ref{prop:zmax_factors} states that  this property is preserved every time we pass to a finite-index or normal subgroup, therefore  $P_1$ is not contained in any proper $\zmax$-factor of $A$. By Part 4 of Proposition~\ref{prop:zmax_factors}, $P_1$ contains an element $w$ which is not contained in any proper $\zmax$-factor of $A$. In particular $w$ is not contained in a proper free factor of $A$, and Lemma~\ref{lemma:single-splitting-2} tells us that the splitting $S$ is the only free splitting of $F_N$ which is $P_1$-invariant. Therefore the centralizer of $P_1$ also preserves $S$.

Now let $\Phi$ be any element of the centralizer of $P_1$. Then by the above $\Phi \in \Stab_{\Gamma}(S)$. We claim that the image $\Phi_{|A}$ of $\Phi$ in $\Out(A)$ is trivial. To see this, let $w$ be the above element of $P_1$ that is not contained in any $\zmax$-factor of $A$. As $\Phi$ commutes with the twist given by $w$, the automorphism $\Phi_{|A}$ preserves the conjugacy class of the subgroup generated by $w$ (Lemma~\ref{lemma:twistor}). Then $\Phi_{|A}$ is finite-order in $\Out(A)$ by Proposition~\ref{p:zmax_rigid}. As $\Phi \in \ia$ the restriction $\Phi_{|A}$ is contained in $\mathrm{IA}(A,\mathbb{Z}/3\mathbb{Z})$, which is torsion-free, so $\Phi_{|A}$ is trivial.  Hence $C_{\Gamma}(P_1)$ is contained in the group of twists of the splitting $S$. As $P_1$ is a nonabelian group of left twists it follows that $C_{\Gamma}(P_1)$ is contained in the group of right twists.
\end{proof}

\subsection{Characterizing stabilizers of nonseparating free splittings}\label{sec:converse}

We now provide a converse statement to Proposition~\ref{prop:property-satisfied}.

\begin{prop}\label{criterion-fix-splitting}
Let $\Gamma\subseteq\ia$ be a subgroup that satisfies Hypothesis~$(H_1)$. Let $H$ be a subgroup of $\Gamma$ which satisfies Property~$(P_\stab)$.
\\ Then $H$ fixes a one-edge nonseparating free splitting of $F_N$.  
\end{prop}

\begin{proof}
We will show that $H$ fixes a one-edge free splitting of $F_N$; the fact that this splitting is nonseparating then follows from the fact that $H$ contains a direct product of $2N-4$ nonabelian free groups (Hypothesis~$2$ from Property~$\Pstab$), while stabilizers of one-edge separating free splittings do not (Theorem~\ref{direct_products_of_free_groups}).

Assume towards a contradiction that $H$ does not fix any free splitting of $F_N$, and let $\calf$ be a maximal $H$-invariant free factor system of $F_N$ (so in particular $H\subseteq\Out(F_N,\calf)$). Then $\calf$ is nonsporadic. For ease of notation, we simply denote by $\FF$ the relative free factor graph $\FF(F_N,\calf)$. As $\calf$ is maximal, Proposition~\ref{prop:maximal-unbounded} tells us that $H$ acts on $\FF$ with unbounded orbits.

Let $K_1$ and $K_2$ be nonabelian free subgroups of $H$ as in Hypothesis~1 from Property~$\Pstab$. We first assume that both $K_1$ and $K_2$ contain a fully irreducible automorphism relative to $\calf$ (which are loxodromic in $\FF$ by \cite[Theorem~A]{Gup} or \cite[Theorem~4.1]{GH}). By Lemma~\ref{loxo-loxo}, the groups $K_1$ and $K_2$ have finite-index subgroups $K_1^0$ and $K_2^0$ that share a common fixed point $\xi$ in $\partial_\infty \FF$. By Proposition~\ref{prop:fix-boundary}, a finite index subgroup of the stabilizer of $\xi$ preserves the homothety class $[T]$ of an arational $(F_N,\calf)$-tree $T$. We can therefore pass to two further finite-index subgroups $K'_1\subseteq K_1$ and $K'_2\subseteq K_2$ which also fix $[T]$.

There is a map $\Stab_{\Gamma}([T])\to\mathbb{R}_+^\ast$ (given by the homothety factor), whose kernel is equal to the isometric stabilizer $\Stab_\Gamma(T)$. We let $P_1:=K'_1\cap\Stab_{\Gamma}(T)$ and $P_2:=K'_2\cap\Stab_{\Gamma}(T)$ be the respective intersections of $K_1'$ and $K_2'$ with this isometric stabilizer. For $i \in \{1,2\}$, the group $P_i$ is nonabelian and normal in $K_i'$ as it is the kernel of a map from $K_i'$ to an abelian group. As $T$ is an arational $(F_N,\calf)$-tree, the first conclusion of Proposition~\ref{prop:stab-arat} implies that $P_1\times P_2$ virtually centralizes an infinite cyclic subgroup of $\Gamma$. This contradicts the first hypothesis from $\Pstab$.  

Up to exchanging the roles of $K_1$ and $K_2$, we can therefore assume that $K_1$ contains no fully irreducible automorphism relative to $\calf$. 
Then $K_1$ does not contain a loxodromic element with respect to the action on $\FF$. Since $H$ has unbounded orbits in $\FF$, Proposition~\ref{product-vs-hyp} implies that $K_1$ has a finite-index subgroup $K_1^0$ that fixes a point in $\partial_\infty \FF$. By the same argument as above, we can pass to a further finite-index subgroup $K'_1$ of $K_1$ that preserves the homothety class of an arational $(F_N,\calf)$-tree $T$. As $K'_1$ contains no fully irreducible automorphism relative to $\calf$, it fixes $T$ up to isometry, not just homothety (see e.g.\ \cite[Proposition~6.2]{GH}). Therefore, Proposition~\ref{prop:stab-arat} implies that either $K_1$ virtually centralizes a subgroup of $\Gamma$ isomorphic to $\mathbb{Z}^2$, or else that $H$ (which is contained in $\Comm_{\Gamma\cap\Out(F_N,\calf)}(K_1)$) does not contain any free abelian subgroup of rank $2N-4$. In the former case, we get a contradiction to Hypothesis~$1$ from Property~$\Pstab$. In the latter case $H$ cannot contain a direct product of $2N-4$ nonabelian free groups, contradicting Hypothesis~$2$ from Property~$\Pstab$.   
\end{proof}

\subsection{Conclusion}\label{sec:ccl}

We are now ready to show that the set of all commensurability classes of $\Gamma$-stabilizers of one-edge nonseparating free splittings of $F_N$ is $\Comm(\Gamma)$-invariant.

\begin{prop}\label{prop:stab-invariant}
Let $\Gamma\subseteq\ia$ be a twist-rich subgroup. Let $\Psi\in\Comm(\Gamma)$.
\\ Then for every one-edge nonseparating free splitting $S$ of $F_N$, there exists a unique one-edge nonseparating free splitting $S'$ of $F_N$ such that $\Psi([\Stab_{\Gamma}(S)])=[\Stab_{\Gamma}(S')]$.
\end{prop}

\begin{proof}
As $\Gamma$ is twist-rich, the $\Gamma$-stabilizers of two distinct one-edge nonseparating free splittings of $F_N$ are not commensurable  in $\Gamma$ (Lemma~\ref{lemma:single-splitting-stabilized}), so $S'$ is unique. 

We now prove existence. Let $f:\Gamma_1\to\Gamma_2$ be an isomorphism between two finite-index subgroups of $\Gamma$ that represents $\Psi$. Proposition~\ref{prop:fi-twist-rich} states that finite-index subgroups of twist-rich groups are twist-rich, so both $\G_1$ and $\G_2$ are twist-rich. By Proposition~\ref{prop:property-satisfied}, the group $\Stab_{\Gamma_1}(S)$ satisfies Property~$(P_\stab)$. As $f$ is an isomorphism, we deduce that $f(\Stab_{\Gamma_1}(S))$ also satisfies Property~$(P_\stab)$. Proposition~\ref{criterion-fix-splitting} implies that there exists a one-edge nonseparating free splitting $S'$ of $F_N$ such that  $f(\Stab_{\Gamma_1}(S))\subseteq\Stab_{\Gamma_2}(S')$. Applying the same argument to $f^{-1}$, we deduce that there exists a one-edge nonseparating free splitting $S''$ such that $$\Stab_{\Gamma_1}(S)\subseteq f^{-1}(\Stab_{\Gamma_2}(S'))\subseteq\Stab_{\Gamma_1}(S'').$$ Lemma~\ref{lemma:single-splitting-stabilized} tells us that $S$ is the unique free splitting invariant under $\Stab_{\Gamma_1}(S)$, so that $S=S''$, and we have equality everywhere. This completes our proof. 
\end{proof}

\section{Characterizing rose-compatibility}\label{sec:edges}

\emph{The goal of the present section is to give an algebraic characterization of when two one-edge nonseparating free splittings of $F_N$ are rose-compatible. This will imply that $\Comm(\Gamma)$ preserves the set of pairs of commensurability classes of stabilizers of adjacent vertices in $\ens$.}
\\
\\
\indent Here two compatible one-edge nonseparating free splittings of $F_N$ are said to be \emph{rose-compatible} if, denoting by $U$ their two-edge refinement, the graph $U/F_N$ is a two-petal rose; they are called \emph{circle-compatible} if $U/F_N$ is a loop with two vertices. 

\indent The general idea will be to use the fact that two one-edge nonseparating free splittings $S$ and $S'$ of $F_N$ are compatible if and only if their common stabilizer does not fix a third one-edge free splitting $S''$. Using the fact that stabilizers of nonseparating free splittings are preserved by the commensurator (as established in the previous section), we will show that this compatibility property is also preserved up to commensuration. We recall that edges in $\ens$ are given by rose-compatibility; distinguishing rose-compatibility from circle-compatibility  for $N\ge 4$ will follow from the fact that the stabilizer of a two-petalled rose in a twist-rich subgroup contains a direct product of $2N-4$ nonabelian free groups whereas the stabilizer of a two-edge loop splitting does not. In rank $3$ to distinguish rose-compatibility from circle-compatibility we will look at maximal free abelian subgroups instead. 

\subsection{The case when $N\ge 4$}

Let $N\ge 4$, and let $\Gamma\subseteq\ia$ be a twist-rich subgroup of $\ia$. We consider the following property of a pair $(K_1,K_2)$ of subgroups of $\Gamma$.

\begin{itemize}
\item[$\Pcomp$] Whenever $K\subseteq\Gamma$ is a subgroup that contains $K_1\cap K_2$ and satisfies $\Pstab$, we either have $K\subseteq K_1$ or $K\subseteq K_2$. In addition $K_1\cap K_2$ contains a direct product of $2N-4$ nonabelian free groups. 
\end{itemize}

\begin{prop}\label{prop:compatibility-all-cases}
Let $N\ge 4$, and let $\Gamma$ be a twist-rich subgroup of $\ia$. Let $S_1$ and $S_2$ be two one-edge nonseparating free splittings of $F_N$, and for every $i\in\{1,2\}$, let $K_i:=\Stab_{\Gamma}(S_i)$.
\\ Then $S_1$ and $S_2$ are rose-compatible if and only if $(K_1,K_2)$ satisfies Property~$\Pcomp$.
\end{prop}

Our proof of Proposition~\ref{prop:compatibility-all-cases} relies on the following lemma, whose proof turns out to have a nice formulation in the sphere model of splittings of $F_N$.

\begin{lemma}\label{lemma:new-splitting}
Let $N\ge 3$, and let $S_1$ and $S_2$ be two noncompatible one-edge free splittings of $F_N$. Then there exists a one-edge free splitting $S$ of $F_N$ which is distinct from both $S_1$ and $S_2$, and fixed by $\Stab_{\ia}(S_1)\cap\Stab_{\ia}(S_2)$.
\end{lemma}

\begin{remark}\label{rk:surgery}
The proof will actually show that every free splitting (corresponding to a sphere) which appears on a surgery path from $S_1$ to $S_2$ is fixed by $\Stab_{\ia}(S_1)\cap\Stab_{\ia}(S_2)$.
\end{remark}

\begin{proof}
Viewing $S_1$ and $S_2$ as spheres in $M_N$, as they are noncompatible there is a nontrivial \emph{surgery sequence} from $S_1$ to $S_2$ (see e.g.\ \cite{Hat} or \cite{HV2}), and there are only finitely many of those. If $\Phi\in\Stab_{\ia}(S_1)\cap\Stab_{\ia}(S_2)$, then the $\Phi$-image of a surgery sequence from $S_1$ to $S_2$ is again a surgery sequence from $S_1$ to $S_2$. In particular, as we are working in $\ia$, every sphere on a surgery sequence from $S_1$ to $S_2$ is fixed by $\Stab_{\ia}(S_1)\cap\Stab_{\ia}(S_2)$.  

Now let $S$ be an essential sphere obtained by a single surgery on $S_1$ towards $S_2$.  The sphere $S$ is disjoint from $S_1$ so is not isotopic to $S_2$ and has strictly fewer intersection circles with $S_2$ than $S_1$, so is not isotopic to $S_1$. By the above, it is fixed by $\Stab_{\ia}(S_1)\cap\Stab_{\ia}(S_2)$. This concludes our proof. 
\end{proof}

\begin{proof}[Proof of Proposition~\ref{prop:compatibility-all-cases}]
We first assume that $S_1$ and $S_2$ are rose-compatible. As $N\ge 4$, the splittings $S_1$ and $S_2$ have a common refinement $U$ such that $U/F_N$ is a rose with $N-2$ petals with nonabelian vertex group (isomorphic to $F_2$). The group of twists on this rose is isomorphic to a direct product of $2N-4$ copies of $F_2$ (with each factor given by the group of twists on a half-edge). Hypothesis~$(H_2)$ on $\Gamma$ thus ensures that $K_1\cap K_2$ contains a direct product of $2N-4$ nonabelian free groups.

Now let $K_0:=K_1\cap K_2$. Then $K_0$ is equal to the stabilizer in $\ia$ of the two-edge common refinement $V$ of $S_1$ and $S_2$ (by Lemma~\ref{lemma:stab-splitting-ia}, the stabilizer of a two-edge free splitting in $\ia$ also fixes each of the two one-edge collapses, without permuting them). Let $K\subseteq\Gamma$ be a group that contains $K_0$ and satisfies $\Pstab$. As $K$ satisfies $\Pstab$, Proposition~\ref{criterion-fix-splitting} ensures that $K$ fixes a nonseparating free splitting $S$ of $F_N$. As $K_0\subseteq K$ we deduce that $S$ is $K_0$-invariant. However, as $\Gamma$ is twist-rich, Lemma~\ref{lemma:single-splitting-stabilized} ensures that the only $K_0$-invariant one-edge free splittings are the collapses of $V$, which are $S_1$ and $S_2$. Therefore $K\subseteq K_1$ or $K\subseteq K_2$. This shows that the pair $(K_1,K_2)$ satisfies $\Pcomp$.

We now assume that $S_1$ and $S_2$ are not rose-compatible. If they are circle-compatible, then $K_1\cap K_2$ does not contain any direct product of $2N-4$ nonabelian free groups (Lemma~\ref{lemma:product-in-circle}), so $(K_1,K_2)$ does not satisfy $\Pcomp$. We now assume that $S_1$ and $S_2$ are not compatible. By Lemma~\ref{lemma:new-splitting}, there exists a one-edge free splitting $S$ of $F_N$, distinct from both $S_1$ and $S_2$, which is fixed by $K_1\cap K_2$. If $S$ is a separating splitting, then $K_1\cap K_2$ does not contain any direct product of $2N-4$ nonabelian free groups (Theorem~\ref{direct_products_of_free_groups}), so $(K_1,K_2)$ does not satisfy $\Pcomp$. If $S$ is a nonseparating splitting, we let $K:=\Stab_{\Gamma}(S)$. Proposition~\ref{prop:property-satisfied} ensures that $K$ satisfies $\Pstab$, and we have $K_1\cap K_2\subseteq K$, however Lemma~\ref{lemma:single-splitting-stabilized} tells us that $S$ is the only invariant free splitting of $K$, so that $K$ is neither contained in $K_1$ nor in $K_2$. Therefore $(K_1,K_2)$ does not satisfy $\Pcomp$.
\end{proof}

\subsection{Subgroups of $\Out(F_3)$ that contain a power of every twist}

Let $N=3$, and let $\Gamma\subseteq\mathrm{IA}_3(\mathbb{Z}/3\mathbb{Z})$ be a subgroup such that every twist has a power in $\Gamma$. We consider the following property of a pair $(K_1,K_2)$ of subgroups of $\Gamma$.

\begin{itemize}
\item[$\Ppcomp$] The group $K_1\cap K_2$ is isomorphic to $\mathbb{Z}^3$. In addition, whenever $K\subseteq\Gamma$ is a subgroup that contains $K_1\cap K_2$ and satisfies $\Pstab$, we either have $K\subseteq K_1$ or $K\subseteq K_2$.  
\end{itemize}

\begin{lemma}\label{lemma:vcd-f3}
The stabilizer in $\mathrm{IA}_3(\mathbb{Z}/3\mathbb{Z})$ of a free splitting $S$ such that $S/F_3$ is a two-petal rose is isomorphic to $\mathbb{Z}^3$. The stabilizer of a one-edge separating free splitting of $F_3$ or of a free splitting $S$ such that $S/F_3$ is a two-edge loop does not contain any free abelian subgroup of rank $3$.
\end{lemma}

\begin{proof}
It follows from \cite{Lev} that the stabilizer of a two-petal rose in $\mathrm{IA}_3(\mathbb{Z}/3\mathbb{Z})$ is isomorphic to $\mathbb{Z}^3$, the stabilizer of a two-edge loop is isomorphic to $\mathbb{Z}^2$, and the stabilizer $\Stab_{\mathrm{IA}_3(\mathbb{Z}/3\mathbb{Z})}(S)$ of a separating free splitting $S$ of the form $F_2\ast\mathbb{Z}$ fits into a short exact sequence $$1\to F_2\to \Stab_{\mathrm{IA}_3(\mathbb{Z}/3\mathbb{Z})}(S)\to\Out(F_2)\to 1,$$ from which the result follows.
\end{proof}

\begin{prop}\label{prop:compatibility-3}
Let $\Gamma\subseteq\mathrm{IA}_3(\mathbb{Z}/3\mathbb{Z})$ be a subgroup which contains a power of every twist. Let $S_1$ and $S_2$ be two one-edge nonseparating free splittings of $F_3$, and for every $i\in\{1,2\}$, let $K_i:=\Stab_{\Gamma}(S_i)$.
\\ Then $S_1$ and $S_2$ are rose-compatible if and only if $(K_1,K_2)$ satisfies $\Ppcomp$.
\end{prop}

\begin{proof}
The proof is the same as the proof of Proposition~\ref{prop:compatibility-all-cases}, using Lemma~\ref{lemma:vcd-f3} instead of maximal direct products of free groups to distinguish nonseparating free splittings from separating ones, and rose-compatibility from circle-compatibility, and Lemma~\ref{lemma:single-splitting-stabilized-2} instead of Lemma~\ref{lemma:single-splitting-stabilized}. 

We first assume that $S_1$ and $S_2$ are rose-compatible. Lemma~\ref{lemma:vcd-f3} ensures that $K_1\cap K_2$ is isomorphic to $\mathbb{Z}^3$. Let $K_0:=K_1\cap K_2$, and let $K\subseteq\Gamma$ be a group that contains $K_0$ and satisfies $\Pstab$. As $K$ satisfies $\Pstab$, Proposition~\ref{criterion-fix-splitting} ensures that $K$ fixes a one-edge nonseparating free splitting $S$ of $F_3$. As $K_0\subseteq K$ we deduce that $S$ is $K_0$-invariant. As every twist has a power contained in $\Gamma$, Lemma~\ref{lemma:single-splitting-stabilized-2} ensures that $S$ is either equal to $S_1$ or to $S_2$. Therefore $K\subseteq K_1$ or $K\subseteq K_2$. This shows that the pair $(K_1,K_2)$ satisfies $\Ppcomp$.

We now assume that $S_1$ and $S_2$ are not rose-compatible. If they are circle-compatible, then $K_1\cap K_2$ does not contain any free abelian subgroup of rank $3$ (Lemma~\ref{lemma:vcd-f3}), so $(K_1,K_2)$ does not satisfy $\Ppcomp$. We now assume that $S_1$ and $S_2$ are not compatible. By Lemma~\ref{lemma:new-splitting}, there exists a one-edge free splitting $S$ of $F_N$, distinct from both $S_1$ and $S_2$, which is fixed by $K_1\cap K_2$. If $S$ is a separating splitting, then $K_1\cap K_2$ does not contain any free abelian subgroup of rank $3$ (Lemma~\ref{lemma:vcd-f3}), so $(K_1,K_2)$ does not satisfy $\Ppcomp$. If $S$ is a nonseparating splitting, we let $K:=\Stab_{\Gamma}(S)$. Proposition~\ref{prop:property-satisfied} ensures that $K$ satisfies $\Pstab$, and we have $K_1\cap K_2\subseteq K$, however $S$ is the unique splitting fixed by $K$ so that $K$ is neither contained in $K_1$ nor in $K_2$ (Lemma~\ref{lemma:single-splitting-stabilized-2}). Therefore $(K_1,K_2)$ does not satisfy $\Ppcomp$.
\end{proof}

\subsection{The case of $\mathrm{IA}_3$}

We remind the reader that $\mathrm{IA}_3$ is the kernel of the natural map $\Out(F_3)\to\mathrm{GL}(3,\mathbb{Z})$. Given a finite-index subgroup $\Gamma$ of $\mathrm{IA}_3$, we consider the following property of a pair $(K_1,K_2)$ of subgroups of $\Gamma$.

\begin{itemize}
\item[$\Pscomp$] The group $K_1\cap K_2$ is isomorphic to $\mathbb{Z}$. In addition, whenever $K\subseteq\Gamma$ is a subgroup that contains $K_1\cap K_2$ and satisfies $\Pstab$, we either have $K\subseteq K_1$ or $K\subseteq K_2$. 
\end{itemize}

\begin{lemma}\label{lemma:stab-rose-ia}
Let $S$ be a free splitting of $F_3$ such that the quotient graph $S/F_3$ is a two-petal rose, and suppose that $\{a,b,c\}$ is a free basis of $F_3$ such that $S$ is the common refinement of the splittings $\langle a,b\rangle\ast$ and $\langle a,c\rangle\ast$ (such a basis always exists). 
\\ Then the stabilizer of $S$ in $\mathrm{IA}_3$ is equal to the group of twists about the one-edge separating cyclic splitting $\langle a,b\rangle\ast_{\langle a\rangle}\langle a,c\rangle$; in particular it is isomorphic to $\mathbb{Z}$.
\end{lemma}

\begin{proof}
The stabilizer of $S$ in $\mathrm{IA}_3(\mathbb{Z}/3\mathbb{Z})$ is generated by the Dehn twists $c\mapsto ac$, $c\mapsto ca$ and $b\mapsto ba$. The stabilizer in $\mathrm{IA}_3$ is therefore generated by the partial conjugation $c\mapsto a^{-1}ca$, so the conclusion follows.
\end{proof}

\begin{lemma}\label{lemma:stab-circle-ia}
Let $S$ be a free splitting of $F_3$ such that the quotient graph $S/F_3$ is a two-edge loop. Then the stabilizer of $S$ in $\mathrm{IA}_3$ is trivial.
\end{lemma}

\begin{proof}
There exists a free basis $\{a,b,c\}$ of $F_3$ such that $S$ is the free splitting which is the common refinement of the splittings $F_3=\langle a,b\rangle\ast$ and $F_3=\langle a,cbc^{-1}\rangle\ast$. The stabilizer of $S$ in $\Out(F_3)$ has a finite-index subgroup generated by the Dehn twists $c\mapsto cb$ and $c\mapsto ac$. One then sees that its intersection with $\mathrm{IA}_3$ is trivial. 
\end{proof}

We will need the following variation over Lemma~\ref{lemma:single-splitting-stabilized}.

\begin{lemma}\label{lemma:single-splitting-stabilized-ia}
Let $\Gamma$ be a finite-index subgroup of $\mathrm{IA}_3$. Let $S$ be a free splitting of $F_3$ whose quotient graph $S/F_3$ is a two-petal rose, and let $K$ be the stabilizer of $S$ in $\Gamma$. Let $S'$ be a one-edge nonseparating free splittings of $F_3$ which is $K$-invariant.
\\ Then $S'$ is a collapse of $S$.
\end{lemma}

\begin{proof}
Since every $K$-invariant free splitting of $F_3$ is invariant under the $\mathrm{IA}_3$-stabilizer of $S$, we can assume without loss of generality that $\Gamma=\mathrm{IA}_3$. There exists a free basis $\{a,b,c\}$ of $F_3$ such that $S$ is the two-edge refinement of the splittings $\langle a,b\rangle\ast$ and $\langle a,c\rangle\ast$. By Lemma~\ref{lemma:stab-rose-ia}, the group $K$ is equal to the group of twists about the one-edge separating cyclic splitting $U$ equal to $\langle a,b\rangle\ast_{\langle a\rangle}\langle a,c\rangle$. By Lemma~\ref{twist-compatible}, all free splittings of $F_3$ which are $K$-invariant are compatible with $U$. As the only nonseparating free splittings of $F_3$ compatible with $U$ are the two one-edge collapses of $S$ (there is only one way to blow-up each vertex), the conclusion follows.
\end{proof}

We will also need the following extension of Lemma~\ref{lemma:new-splitting} (valid in any rank $N$). 

\begin{lemma}\label{lemma:new-splitting-2}
Let $S_1$ and $S_2$ be two  one-edge nonseparating free splittings of $F_N$, which are the Bass--Serre trees of two decompositions $F_N=A_1\ast$ and $F_N=A_2\ast$, respectively.  Assume that $S_1$ and $S_2$ are noncompatible, and let $K:=\Stab_{\ia}(S_1)\cap\Stab_{\ia}(S_2)$. 
\\ Then there exists a $K$-invariant one-edge free splitting $S$ of $F_N$ which is distinct from both $S_1$ and $S_2$ and from every separating free splitting of the form $A_1\ast\mathbb{Z}$ or $A_2\ast\mathbb{Z}$.
\end{lemma}

\begin{proof}
Assume towards a contradiction that all $K$-invariant free splittings of $F_N$ have one of the forms in the statement. In view of Remark~\ref{rk:surgery}, this implies in particular that all spheres on a surgery sequence from $S_1$ to $S_2$ correspond to splittings of the form $A_1\ast\mathbb{Z}$ or $A_2\ast\mathbb{Z}$.

If all the splittings on surgery sequences from $S_1$ to $S_2$ are of the form $A_1\ast\mathbb{Z}$, then one of those (call it $S$) is compatible with $S_2$. But then $S_2$ is obtained from $S$ by blowing up the vertex with vertex group $A_1$ and collapsing the edge coming from $S$, while $S_1$ is obtained from $S$ by blowing up the vertex with vertex group $\mathbb{Z}$ and collapsing the edge coming from $S$. This implies that $S_1$ and $S_2$ are compatible, a contradiction. Likewise, if all the splittings on surgery sequences from $S_1$ to $S_2$ are of the form $A_2\ast\mathbb{Z}$, then we get a contradiction.

In the remaining case, we can find a splitting of the form $A_1\ast\langle a_1\rangle$ and a splitting of the form $A_2\ast\langle a_2\rangle$ which follow each other in the surgery sequence and are therefore compatible. But then their common refinement is of the form $\langle a_1\rangle\ast A\ast\langle a_2\rangle$, and both $S_1$ and $S_2$ are compatible with it (as seen by blowing up the vertices with vertex groups $\langle a_1\rangle$ and $\langle a_2\rangle$, respectively). Again this proves that $S_1$ and $S_2$ are compatible, a contradiction. 
\end{proof}

\begin{prop}\label{prop:compatibility-ia} 
Let $\Gamma$ be a finite-index subgroup of $\mathrm{IA}_3$. Let $S_1$ and $S_2$ be two one-edge nonseparating free splittings of $F_3$, and for every $i\in\{1,2\}$, let $K_i:=\Stab_{\Gamma}(S_i)$.
\\ Then $S_1$ and $S_2$ are rose-compatible if and only if $(K_1,K_2)$ satisfies $\Pscomp$.
\end{prop}

\begin{proof}
The proof is similar to the proofs of Propositions~\ref{prop:compatibility-all-cases} and~\ref{prop:compatibility-3}. Let $A_1$ and $A_2$ be corank one free factors of $F_3$ such that for every $i\in\{1,2\}$, the tree $S_i$ is the Bass--Serre tree of the decomposition $F_3=A_i\ast$.

We first assume that $S_1$ and $S_2$ are rose-compatible. Lemma~\ref{lemma:stab-rose-ia} ensures that $K_1\cap K_2$ is isomorphic to $\mathbb{Z}$. Let $K_0:=K_1\cap K_2$, and let $K\subseteq\Gamma$ be a group that contains $K_0$ and satisfies $\Pstab$. As $K$ satisfies $\Pstab$, Proposition~\ref{criterion-fix-splitting} ensures that $K$ fixes a one-edge nonseparating free splitting $S$ of $F_3$. As $K_0\subseteq K$ we deduce that $S$ is $K_0$-invariant. Lemma~\ref{lemma:single-splitting-stabilized-ia} therefore ensures that $S$ is equal to either $S_1$ or $S_2$. Therefore $K\subseteq K_1$ or $K\subseteq K_2$. This shows that the pair $(K_1,K_2)$ satisfies $\Pscomp$.

We now assume that $S_1$ and $S_2$ are not rose-compatible. If they are circle-compatible, then $K_1\cap K_2$ is trivial (Lemma~\ref{lemma:stab-rose-ia}), so $(K_1,K_2)$ does not satisfy $\Pscomp$. We now assume that $S_1$ and $S_2$ are not compatible. 

 We claim that there exists a one-edge nonseparating free splitting $S$ of $F_3$, distinct from both $S_1$ and $S_2$, which is fixed by $K_1\cap K_2$. Indeed, by Lemma~\ref{lemma:new-splitting-2}, there exists a one-edge free splitting $S'$ of $F_3$, distinct from both $S_1$ and $S_2$, which is fixed by $K_1\cap K_2$; in addition, if $S'$ is separating, then we can assume that $S'$ is the Bass--Serre tree of a decomposition $F_3=C\ast\mathbb{Z}$ where $C$ is not conjugate to any $A_i$. If $S'$ is nonseparating, then we are done by letting $S=S'$. If $S'$ is separating, then 
 we are done by letting $S'$ be the nonseparating splitting $F_3=C\ast$, as any automorphism that fixes $C \ast \mathbb{Z}$ also preserves the conjugacy class of $C$. 

We then let $K:=\Stab_{\Gamma}(S)$. Proposition~\ref{prop:property-satisfied} ensures that $K$ satisfies $\Pstab$, and we have $K_1\cap K_2\subseteq K$. However, Lemma~\ref{lemma:single-splitting-stabilized-ia} ensures that $K$ is neither contained in $K_1$ nor in $K_2$. Therefore $(K_1,K_2)$ does not satisfy $\Pscomp$.
\end{proof}

\section{Conclusion}\label{sec:conclusion}

\emph{In this last section, we complete the proof of our main theorem.}

\begin{theo}\label{theo:main}
Let $N\ge 4$, and let $\Gamma\subseteq\ia$ be a twist-rich subgroup. Then any isomorphism $f\colon H_1 \to H_2$ between two finite index subgroups of $\G$ is given by conjugation by an element of $\Comm_{\Out(F_N)}(\G)$ and the natural map \[\Comm_{\Out(F_N)}(\Gamma)\to\Comm(\Gamma) \] is an isomorphism. 
\end{theo}

\begin{proof}
If $S$ and $S'$ are two different one-edge nonseparating free splittings of $F_N$, then $\Stab_\Gamma(S)$ and $\Stab_\Gamma(S')$ are not commensurable (Lemma~\ref{lemma:single-splitting-stabilized}). Proposition~\ref{prop:stab-invariant} shows that the collection $\cali$ of all commensurability classes of $\Gamma$-stabilizers of one-edge nonseparating free splittings of $F_N$ is $\Comm(\Gamma)$-invariant. Proposition~\ref{prop:compatibility-all-cases} shows that the collection $\calj$ of all pairs $([\Stab_\Gamma(S)],[\Stab_\Gamma(S')])$, where $S$ and $S'$ are two rose-compatible one-edge nonseparating free splittings of $F_N$, is also $\Comm(\Gamma)$-invariant. As the natural morphism $\Out(F_N)\to\Aut(\ens)$ is an isomorphism (Theorem~\ref{ens-automorphisms}), the conclusion follows from Proposition~\ref{prop:blueprint}.
\end{proof}

The proof of Theorem~1 from the introduction follows from the fact that a subgroup $\G \subseteq \Out(F_N)$ containing a term of the Andreadakis--Johnson filtration or a power of every twist is twist-rich (Proposition~\ref{prop:example} and Proposition~\ref{prop:example2}). The second theorem from the introduction is the following:

\begin{theo}
Let $\Gamma$ be either $IA_3$ or a subgroup of $\Out(F_3)$ such that every twist has a power contained in $\Gamma$. Then any isomorphism $f\colon H_1 \to H_2$ between two finite index subgroups of $\G$ is given by conjugation by an element of $\Comm_{\Out(F_3)}(\G)$ and the natural map \[\Comm_{\Out(F_3)}(\Gamma)\to\Comm(\Gamma) \] is an isomorphism. 
\end{theo}

\begin{proof}
The proof is the same as the proof of Theorem~\ref{theo:main}, using Proposition~\ref{prop:compatibility-3} or~\ref{prop:compatibility-ia} instead of Proposition~\ref{prop:compatibility-all-cases}. 
\end{proof}

\bibliographystyle{alpha}
\bibliography{commensurator-bib}

\newcommand{\etalchar}[1]{$^{#1}$}
\begin{thebibliography}{CdCMT15}

\bibitem[AGK{\etalchar{+}}18]{AGKMTW}
J.~Aramayona, T.~Ghaswala, A.E. Kent, A.~Mcleay, J.~Tao, and R.R. Winarski.
\newblock Big {T}orelli groups: generation and commensuration.
\newblock {\em arXiv:1810.03453}, 2018.

\bibitem[AS11]{AS}
J.~Aramayona and J.~Souto.
\newblock Automorphisms of the graph of free splittings.
\newblock {\em Michigan Math. J.}, 60(3):483--493, 2011.

\bibitem[BB10]{MR2736164}
L.~Bartholdi and O.~Bogopolski.
\newblock On abstract commensurators of groups.
\newblock {\em J. Group Theory}, 13(6):903--922, 2010.

\bibitem[BBon]{BB}
M.~Bestvina and M.R. Bridson.
\newblock Rigidity of the complex of free factors.
\newblock in preparation.

\bibitem[BDR18]{BDR}
J.~Bavard, S.~Dowdall, and K.~Rafi.
\newblock Isomorphisms {B}etween {B}ig {M}apping {C}lass {G}roups.
\newblock {\em Int. Math. Res. Notices}, 2018.

\bibitem[BF14]{BF}
M.~Bestvina and M.~Feighn.
\newblock Hyperbolicity of the free factor complex.
\newblock {\em Adv. Math.}, 256:104--155, 2014.

\bibitem[BM04]{BM2}
T.E. Brendle and D.~Margalit.
\newblock Commensurations of the {J}ohnson kernel.
\newblock {\em Geom. Topol.}, 8:1361--1384, 2004.

\bibitem[BM17]{BM}
T.E. Brendle and D.~Margalit.
\newblock Normal subgroups of mapping class groups and the metaconjecture of
  {I}vanov.
\newblock {\em arXiv:1710.08929}, 2017.

\bibitem[Bor66]{Bor}
A.~Borel.
\newblock Density and maximality of arithmetic subgroups.
\newblock {\em J. Reine angew. Math.}, 224:78--89, 1966.

\bibitem[BPSon]{BPS}
M.R. Bridson, A.~Pettet, and J.~Souto.
\newblock The abstract commensurator of the {J}ohnson kernels.
\newblock in preparation.

\bibitem[BR15]{BR}
M.~Bestvina and P.~Reynolds.
\newblock The boundary of the complex of free factors.
\newblock {\em Duke Math. J.}, 164(11):2213--2251, 2015.

\bibitem[BV00]{BV}
M.~Bridson and K.~Vogtmann.
\newblock Automorphisms of automorphism groups of free groups.
\newblock {\em J. Algebra}, 229(2):785--792, 2000.

\bibitem[BV01]{BV2}
M.~Bridson and K.~Vogtmann.
\newblock The symmetries of outer space.
\newblock {\em Duke Math. J.}, 106(2):391--409, 2001.

\bibitem[CdCMT15]{CCMT}
P.-E. Caprace, Y.~de~Cornulier, N.~Monod, and R.~Tessera.
\newblock Amenable hyperbolic groups.
\newblock {\em J. Eur. Math. Soc.}, 11:2903--2947, 2015.

\bibitem[CL95]{CL}
M.M. Cohen and M.~Lustig.
\newblock Very small group actions on $\mathbb{{R}}$-trees and {D}ehn twist
  automorphisms.
\newblock {\em Topology}, 34(3):575--617, 1995.

\bibitem[CL99]{CL2}
M.M. Cohen and M.~Lustig.
\newblock The conjugacy problem for {D}ehn twist automorphisms of free groups.
\newblock {\em Comment. Math. Helv.}, 74(2):179--200, 1999.

\bibitem[CV86]{CV}
M.~Culler and K.~Vogtmann.
\newblock Moduli of graphs and automorphisms of free groups.
\newblock {\em Invent. Math.}, 84(1):91--119, 1986.

\bibitem[DGO17]{DGO}
F.~Dahmani, V.~Guirardel, and D.~Osin.
\newblock {\em Hyperbolically embedded subgroups and rotating families in
  groups acting on hyperbolic spaces}, volume 245 of {\em Mem. Amer. Math.
  Soc.}
\newblock 2017.

\bibitem[DV96]{DV}
W.~Dicks and E.~Ventura.
\newblock {\em The group fixed by a family of injective endomorphisms of a free
  group}, volume 195 of {\em Contemp. Math.}
\newblock Amer. Math. Soc., Providence, RI, 1996.

\bibitem[FH07]{FH}
B.~Farb and M.~Handel.
\newblock Commensurations of $\text{{O}ut}({F}_n)$.
\newblock {\em Publ. Math. IHES}, 105(1):1--48, 2007.

\bibitem[FH11]{FeH}
M.~Feighn and M.~Handel.
\newblock Abelian subgroups of $\text{{O}ut}({F}_n)$.
\newblock {\em Geom. Topol.}, 5(1):39--106, 2011.

\bibitem[FI14]{FI}
B.~Farb and N.V. Ivanov.
\newblock Torelli buildings and their automorphisms.
\newblock {\em arXiv:1410.6223}, 2014.

\bibitem[GH17]{GH1}
V.~Guirardel and C.~Horbez.
\newblock Algebraic laminations for free products and arational trees.
\newblock {\em arXiv:1709.05664}, 2017.

\bibitem[GH19]{GH}
V.~Guirardel and C.~Horbez.
\newblock Boundaries of relative factor graphs and subgroup classification for
  automorphisms of free products.
\newblock {\em arXiv:1901.05046}, 2019.

\bibitem[GL]{GL}
V.~Guirardel and G.~Levitt.
\newblock in preparation.

\bibitem[GL95]{GabL}
D.~Gaboriau and G.~Levitt.
\newblock The rank of actions on $\mathbb{{R}}$-trees.
\newblock {\em Ann. Scient. Ec. Norm. Sup.}, 28(5):549--570, 1995.

\bibitem[GL15]{GL-aut}
V.~Guirardel and G.~Levitt.
\newblock Splittings and automorphisms of relatively hyperbolic groups.
\newblock {\em Groups Geom. Dyn.}, 9(2):599--663, 2015.

\bibitem[GL17]{GL-jsj}
V.~Guirardel and G.~Levitt.
\newblock J{SJ} decompositions of groups.
\newblock {\em Ast\'{e}risque}, (395):vii+165, 2017.

\bibitem[Gro87]{Gro}
M.~Gromov.
\newblock Hyperbolic groups.
\newblock In S.M. Gersten, editor, {\em Essays in Group Theory}, volume~8 of
  {\em MSRI Series}, pages 75--263. Springer, 1987.

\bibitem[GS18]{GS}
V.~Guirardel and A.~Sale.
\newblock Vastness properties of automorphism groups of {RAAG}s.
\newblock {\em J. Topol.}, 11(1):30--64, 2018.

\bibitem[Gui04]{Gui}
V.~Guirardel.
\newblock Limit groups and groups acting freely on $\mathbb{R}^n$--trees.
\newblock {\em Geom. Topol.}, 8(3):1427--1470, 2004.

\bibitem[Gup18]{Gup}
R.~Gupta.
\newblock Loxodromic elements for the relative free factor complex.
\newblock {\em Geom. Dedicata}, 196:91--121, 2018.

\bibitem[Ham14]{Ham}
U.~Hamenstädt.
\newblock The boundary of the free splitting graph and of the free factor
  graph.
\newblock {\em arXiv:1211.1630v5}, 2014.

\bibitem[Hat95]{Hat}
A.~Hatcher.
\newblock Homological stability for automorphism groups of free groups.
\newblock {\em Comm. Math. Helv.}, 70(1):39--62, 1995.

\bibitem[Hen18]{Hen}
S.~Hensel.
\newblock Rigidity and flexibility for handlebody groups.
\newblock {\em Comment. Math. Helv.}, 93(2):335--358, 2018.

\bibitem[HM13]{HM5}
M.~Handel and L.~Mosher.
\newblock Subgroup decomposition in $\mathrm{Out}({F}_n)$ {P}art {II}: {A}
  relative {K}olchin theorem.
\newblock {\em arXiv:1302.2379}, 2013.

\bibitem[HM14]{HM3}
M.~Handel and L.~Mosher.
\newblock Relative free splitting and free factor complexes {I}:
  {H}yperbolicity.
\newblock {\em arXiv:1407.3508v1}, 2014.

\bibitem[Hor14]{Hor0}
C.~Horbez.
\newblock The {T}its alternative for the automorphism group of a free product.
\newblock {\em arXiv:1408.0546v2}, 2014.

\bibitem[HV04]{HV2}
A.~Hatcher and K.~Vogtmann.
\newblock Homology stability for outer automorphism groups of free groups.
\newblock {\em Algebr. Geom. Topol.}, 4:1253--1272, 2004.

\bibitem[HW15]{HW}
C.~Horbez and R.D. Wade.
\newblock Automorphisms of graphs of cyclic splittings of free groups.
\newblock {\em Geom. Dedic.}, 178:171--187, 2015.

\bibitem[Iva97]{Iva}
N.V. Ivanov.
\newblock Automorphisms of complexes of curves and {T}eichmüller spaces.
\newblock {\em Int. Math. Res. Not.}, 14:651--666, 1997.

\bibitem[Khr90]{Khr}
D.G. Khramtsov.
\newblock Completeness of groups of outer automorphisms of free groups.
\newblock In {\em Group-theoretic investigations (Russian)}, pages 128--143,
  Sverdlovsk, 1990. Akad. Nauk SSSR Ural. Otdel.

\bibitem[Kid13]{Kid}
Y.~Kida.
\newblock The co-{H}opfian property of the {J}ohnson kernel and the {T}orelli
  group.
\newblock {\em Osaka J. Math}, 50:309--337, 2013.

\bibitem[KL11]{KL}
I.~Kapovich and M.~Lustig.
\newblock Stabilizers of $\mathbb{R}$-trees with free isometric actions of
  ${F}_{N}$.
\newblock {\em J. Group Theory}, 14(5):673--694, 2011.

\bibitem[Lev05]{Lev}
G.~Levitt.
\newblock Automorphisms of hyperbolic groups and graphs of groups.
\newblock {\em Geom. Dedic.}, 114(1):49--70, 2005.

\bibitem[LL03]{LL}
G.~Levitt and M.~Lustig.
\newblock Irreducible automorphisms of {$F_n$} have north-south dynamics on
  compactified outer space.
\newblock {\em J. Inst. Math. Jussieu}, 2(1):59--72, 2003.

\bibitem[Mar77]{Mar}
G.A. Margulis.
\newblock Discrete groups of motions of manifolds of non-positive curvature.
\newblock {\em Transl. of AMS}, 109:33--45, 1977.

\bibitem[MO10]{MO}
A.~Minasyan and D.~Osin.
\newblock Normal automorphisms of relatively hyperbolic groups.
\newblock {\em Trans. Amer. Math. Soc.}, 362(11):6079--6103, 2010.

\bibitem[Mos68]{Mos}
G.D. Mostow.
\newblock Quasi-conformal mappings in $n$-space and the rigidity of hyperolic
  space forms.
\newblock {\em Inst. Hautes Etudes Sci. Publ. Math.}, 34:53--104, 1968.

\bibitem[Mos73]{Mos2}
G.D. Mostow.
\newblock {\em Strong rigidity of locally symmetric spaces}, volume~78 of {\em
  Annals of Mathematics Studies}.
\newblock Princeton University Press, Princeton, NJ; University of Tokyo Press,
  Tokyo, 1973.

\bibitem[MV04]{MV}
A.~Martino and E.~Ventura.
\newblock Fixed subgroups are compressed in free groups.
\newblock {\em Comm. Algebra}, 32(10):3921--3935, 2004.

\bibitem[Pan14]{Pan}
S.~Pandit.
\newblock The complex of non-separating embedded spheres.
\newblock {\em Rocky Mountain J. Math.}, 44(6):2029--2054, 2014.

\bibitem[Pau88]{Pau}
F.~Paulin.
\newblock Topologie de {G}romov équivariante, structures hyperboliques et
  arbres réels.
\newblock {\em Invent. Math.}, 94(1):53--80, 1988.

\bibitem[Pra73]{Pra}
G.~Prasad.
\newblock Strong rigidity of $\mathbf{Q}$-rank $1$ lattices.
\newblock {\em Invent. Math.}, 21:255--286, 1973.

\bibitem[Sta71]{Sta1}
J.R. Stallings.
\newblock {\em Group theory and three-dimensional manifolds}, volume~4 of {\em
  Yale Mathematical Monographs}.
\newblock Yale University Press, New Haven, Conn. - London, 1971.

\bibitem[Sta99]{Sta}
J.R. Stallings.
\newblock Whitehead graphs on handlebodies.
\newblock In {\em Geometric group theory down under (Canberra, 1996)}, pages
  317--330, Berlin, 1999. de Gruyter.

\bibitem[Whi36]{Whi}
J.H.C Whitehead.
\newblock On certain sets of elements in a free group.
\newblock {\em Proc. London Math. Soc.}, 41:48--56, 1936.

\end{thebibliography}

\begin{flushleft}
Camille Horbez\\
CNRS\\ 
Laboratoire de Math\'ematique d'Orsay\\
Univ.\ Paris-Sud, CNRS, Universit\'e Paris-Saclay\\ 
F-91405 ORSAY\\
\emph{e-mail: }\texttt{camille.horbez@math.u-psud.fr}
\end{flushleft}
~
\begin{flushleft}
Richard D.\ Wade\\
Mathematical Institute\\
University of Oxford\\
Oxford OX2 6GG\\
\emph{e-mail: }\texttt{wade@maths.ox.ac.uk} 
\end{flushleft}

\end{document}